\documentclass{amsart}

\usepackage[hidelinks]{hyperref}

\usepackage[utf8]{inputenc}
\usepackage{amsmath}
\usepackage{amsfonts}
\usepackage{amssymb}
\usepackage{todonotes}
\usepackage{amsthm}
\usepackage{todonotes}
\usepackage{algorithm}
\usepackage{algpseudocode}
\usepackage{multirow}
\usepackage{natbib}

\newtheorem{theorem}{Theorem}
\newtheorem{proposition}{Proposition}
\theoremstyle{remark}
\newtheorem*{remark}{Remark}

\begin{document}
\title[Approximation algorithm for RCPSP with NPV objective]{Approximation algorithm for resource-constrained project scheduling problems with net present value objective}\thanks{Supported by grants Fondecyt 1200809 and FONDEF ID19I10164 from ANID.}
\author{Rodrigo A. Carrasco}
\address{Institute of Mathematical and Computational Engineering \& School of Engineering, Pontificia Universidad Cat\'olica de Chile, Santiago, Chile}
\email{rax@ing.puc.cl}
\author{Diego Fuentes}
\address{Faculty of Engineering and Sciences, Universidad Adolfo Ibáñez, Santiago, Chile}
\email{diegfuentes@alumnos.uai.cl}
\author{Eduardo Moreno}
\address{Faculty of Engineering and Sciences, Universidad Adolfo Ibáñez, Santiago, Chile}
\email{eduardo.moreno@uai.cl}

\date{}


\begin{abstract}
    Resource-constrained project scheduling problems (RCPSP) are at the heart of many production planning problems across a plethora of applications. Although the problem has been studied since the early 1960s, most developments and test instances are limited to problems with less than 300 jobs, far from the thousands present in real-life scenarios. Furthermore, the RCPSP with discounted cost (DC) is critical in many of these settings, which require decision makers to evaluate the net present value of the finished tasks, but the non-linear cost function makes the problem harder to solve or analyze.
    
    In this work, we propose a novel approximation algorithm for the RCPSP-DC. Our main contribution is that, through the use of geometrically increasing intervals, we can construct an approximation algorithm, keeping track of precedence constraints, usage of multiple resources, and time requirements. To our knowledge, this is the first approximation algorithm for this problem. Finally, through experimental analysis over real instances, we report the empirical performance of our approach, showing that our technique allows us to solve sizeable underground mining problems within reasonable time frames and gaps much smaller than the theoretically computed ones.
\end{abstract}
\maketitle 

%


\section{Introduction}

Resource-constrained project scheduling problems (RCPSP) are one of the flagship problems in the scheduling literature. RCPSP consist of scheduling a set of jobs over time, considering precedences between them, and subject to resource consumption limits per time period. The precedence constraints refer to subsets of jobs that must be finished before starting a given job. Resource constraints consider a set of resources consumed during each job's processing time, with a given resource limit per time period. In the classic version of RCPSP, the objective function is to minimize the makespan, that is, the total time required to complete all jobs. We refer the reader to~\citet{HARTMANN20101} for a comprehensive survey of RCPSP and its different variants and extensions and its recent update  at~\citet{HARTMANN20221}.

This work considers the common variant of RCPSP with discounted costs (RCPSP-DC), in which the objective function is the net present value (NPV) of the processed tasks or jobs. This objective is required in many applications, and particularly it is the one required for mining operations. In this setting, each job has a profit, and the objective is to maximize the NPV of the scheduled jobs. More precisely, the nominal profit of each job is discounted over time by a \emph{discount factor} that penalizes future revenues. Note that in the presence of negative profits (which represent costs), not all jobs need to be scheduled to maximize the NPV. 

In this paper, we study RCPSP with NPV objective with a large number of periods (thousands of time periods) and jobs (also of the order of thousands of jobs). A key insight from our approach is that aggregating time periods in a specific way reduces the problem size; however, the obtained algorithm allows us to derive a near-optimal solution to the original problem and prove its approximation bound. In particular, we study the idea of aggregating the time into geometrically increasing time intervals. The intuition behind our approach is that when NPV is present as the objective function, changes in the order of jobs early in the schedule have a much more significant effect than changes in the ones at the end of the schedule. Hence, a higher time resolution is needed at the beginning of the problem, whereas by the end, one can relax the resolution requirements by aggregating more time steps together. Through this procedure, the time aggregation compensates for the effect of the discount rate on the NPV objective, reducing the number of required time periods, but without losing the tractability of the approximation factor.

Furthermore, under some reasonable assumptions for the intended application, we show that this aggregation allows constructing a $\gamma$-approximation of the problem, where $\gamma$ depends on the discount factor and the time horizon but not on the granularity of the time of the problem. To our knowledge, this is the first approximation algorithm for the RCPSP-DC with a known performance guarantee. Moreover, we show that this approximation algorithm performs very well on classic and large instances of RCPSP-DC problems, particularly on instances from underground mine planning.

\subsection{Previous work on RCPSP-DC}


Mixed integer programming (MIP) formulations for RCPSP with NPV objectives appear in the literature of many applications and under different names. In the context of RCPSP problems, the first formal definition appears in \citet{vanhoucke2001maximizing} and is later extended in \citet{vanhoucke2010scatter}, where a local search metaheuristic is proposed to solve the problem. Motivated by the difficulty of solving this problem, most recent developments for this setting are based on metaheuristics. For example, \citet{thiruvady2019maximising} proposes a hybrid ant-colony optimization method that exploits the parallelism of current computational architecture. More recently, \citet{asadujjaman2021immune} and \citet{asadujjaman2022multi} use immune genetic algorithms to solve larger instances, with up to 100 jobs per instance. Other authors have focused on different versions of the RCPSP-DC, particularly for the multi-mode variant where a job has different modes to process, where each mode implies different processing times or resource requirements. For example \citet{machado2022adaptative} proposed a new metaheuristic called ABFO that can solve PSPlib instances (\citep{kolisch1997psplib}) outperforming the genetic algorithm proposed by \citet{leyman2015new} and \citet{leyman2016payment}. Other metaheuristic methods have been proposed for multi-objective multi-mode RCPSP-DC~\citep{bagherinejad2019bi, azimi2021simulation}. In most of these cases, the instances solved for this problem consider only hundreds of jobs scheduled over a hundred periods, with between 1 to 10 resources for each instance. The main drawback with these approaches is that, in real-life mining applications, time intervals and jobs are on the order of thousands or tens of thousands, so new approaches are needed to compute useful solutions.

Additionally, this setting is relevant since RCPSP with NPV objective has become the standard model for long-term planning in the mining industry. Starting from the pioneering work of~\citet{johnson68}, most of the current software and computational tools for {\bf open-pine mine planning} are based on solving variations of RCPSP with NPV objective~\citep{newman09}. In these problems, the number of jobs (blocks to be extracted and processed) goes from thousand to millions, which should be scheduled over dozens of periods (years). However, these jobs do not have a predefined processing time; they consume resources at the scheduled time, and the total number of resources required defines the processing time. Solving problems of this size is only possible through specific decomposition methods, like \citet{chicoisne12} or the BZ Algorithm~\citep{bienstock10, munoz2018}, that solves the LP relaxation of this problem. Combining these methods plus ad-hoc MIP techniques and heuristics for this problem allows for obtaining a near-optimal solution (with optimality gaps $<5\%$) for these problems in a few hours~\citep{rivera2019}. Most of these ideas are available in commercial open-pit mine planning software, like \citet{minemax2017} or \citet{deswik2020}. 

One of the motivations of this paper is to repeat the success story of open-pit problems in underground mine planning. In this setting, the resulting RCPSP-DC problems have fewer jobs but longer processing times and are scheduled at a finer time granularity (days instead of years). Due to this time-resolution requirement, the size of the problems that current approaches can solve is limited to the size of the classical instances of RCPSP-DC~\citep{chowdu2022operations}. In contrast, our approach is focused on solving RCPSP-DC problems over thousand of time periods. To achieve this, instead of aggregating jobs (as in the BZ Algorithm), we aggregate time periods to reduce the size of the problem. A similar idea has been studied by~\citet{hill2022} in the context of underground mining, but considering an arithmetic time aggregation (that is, regular intervals of time of a fixed length). Recent examples of RCPSP-DC models for underground mining are \citet{ogunmodede2022underground} and \citet{nesbitt2021underground}.

Regarding the resources in the RCPSP instances, several different relevant settings exist. The most studied one, known as renewable resources, is when each resource $k$ has a limited constant availability $B_k$ for each time step. Hence, at each time step, this maximum resource availability $B_k$ also limits the number of jobs that can be processed in parallel. Another relevant resource availability setting is the case with {\em cumulative resources}, also known as {\em inventory constraints} (see \citet{Neumann2003} and the references therein). In this setting, there is an availability $B_k$ of resource $k$ for each time period, but the resources that are not used can be stored for the next period, hence accumulating them. This setting is also very relevant in many applications where an inventory of resources is possible, as described in \citet{Chaleshtarti2011}. This work shows theoretical results for the cumulative case, extending the application to the more studied RCPSP with renewable resources.

This paper is focused on optimization methods that guarantee a near-optimal solution for the problem. We propose a geometric time aggregation, where the length of the time intervals grows exponentially to compensate for the discount factor's effect over time, giving each interval a similar weight in terms of the objective function of the problem. The idea of geometric time intervals has been previously studied by~\citet{Carrasco2018:rcasEjor} in the context of single-machine scheduling problems with precedence constraints, where the \emph{speed} of the machine can be adjusted to minimize the weighted completion time of all jobs, but without the requirement of resource constraints as in our current setting, nor the NPV cost function required in our applications.

\subsection{Our contributions}

We provide several contributions to the problem of RCPSP-DC, particularly for large-scale instances of this problem.

\begin{itemize}
    \item We provide a new MIP formulation of the problem that uses a geometric time aggregation. This formulation significantly reduces the model's size, controlling the induced error and providing an upper bound for the problem's optimal value.
    \item We present an approximation algorithm based on solving this MIP formulation and reconstructing a feasible solution for the original problem.
    \item We prove that this algorithm has a bounded performance guarantee under mild assumptions, which depends on the time horizon and the discount factor of the problem but not on its time granularity. Furthermore, this factor is good enough for real instances of the problem ($>0.75$ for problems with a horizon of 2 years and a 10\% annual discount rate).
    \item We extend these results for general RCPSP-DC problems and provide computational experiments to show the excellent performance of the algorithm on classic and large instances of RCPSP-DC problems, particularly on instances from underground mining.
\end{itemize}

The rest of the paper is structured as follows. First, Section~\ref{sec:defs} presents the RCPSP-DC and the notation used along the paper. Then, in Section~\ref{sec:agg}, we show the geometric time aggregation and derive new MIP formulations for the problem based on it. Then, in Section~\ref{sec:algo}, we describe the resulting approximation algorithm and prove its approximation factor. Finally, in Section~\ref{sec:comput}, we present computational results applying the algorithm on benchmark instances.

\section{Problem definition and model formulation}\label{sec:defs}

We now define the notation used in this paper and present a formal definition of the problem.

\subsection{Notation and problem definition}

Let $\mathcal{J} = \{1,\ldots, N\}$ denote the set of jobs to be scheduled along a set $\mathcal{T} = \{1,\ldots, T\}$ of discrete time periods. We also have a set of resources $\mathcal{K} = \{1,\ldots, K\}$ required to realize the different jobs. Each job $j\in\mathcal{J}$ has associated: (1) a \emph{revenue} $f_j\in \mathbb{R}$ which is recollected after \underline{finishing} the job (note that this revenue can be positive or negative); (2) a processing time $p_j \in\mathbb{N}$ which is the number of time periods required to finish the job $j$; (3) an amount of resource $q_{jk}$ of type $k\in\mathcal{K}$ required \underline{at each period} where the job $j$ is active to complete the job; and (4) a set of precedence jobs $\mathcal{P}_j\subset \mathcal{J}$ that must be finished before starting job $j$. We also denote by $R_{kt}$ the amount of (fresh) resources of type $k$ available at time $t$.

Henceforth, a valid schedule for the problem consists of deciding which jobs to process and in which time periods they should be scheduled, so the precedence constraints between jobs and the resource constraints at each time period are satisfied. Furthermore, we assume that the scheduling is non-preemptive (that is, a job being processed cannot be interrupted).

We can compute its total revenue by giving a valid schedule for the problem. This value is given by the net present value of the jobs finished. Let $C_j$ be the completion time of job $j$ (assume $C_j=-\infty$ if the job is not scheduled). Then, the net present value (NPV) of the schedule is given by 
\begin{equation}
    \sum_{j\in \mathcal{J} : C_j\geq 0} \frac{f_j}{(1+r)^{C_j}} \label{eq:npv}
\end{equation}
where $r$ is the discount rate at each time period. Our problem is to find a feasible schedule that maximizes the NPV of the processed jobs.

The net present value is prevalent for project evaluation. Note that a job with negative revenue will not be scheduled under this objective function unless it is required to finish other jobs with higher revenue to compensate for this negative cost. Also, an optimal solution will try to complete jobs with positive revenue as soon as possible and postpone the jobs with negative revenue for later. 

\paragraph{A note on precedence sets}: Given the set of precedences $\mathcal{P} := \{\mathcal{P}_j\}_{j\in\mathcal{J}}$, we can represent this set has a directed graph $G=(\mathcal{J},\mathcal{A})$ where an arc $(j,k)\in\mathcal{A}$ iff $k\in \mathcal{P}_j$.  Note that $G$ is a directed acyclic graph. For a given node (job) $j$, its \emph{transitive closure} is the set of all nodes reachable from $j$ over $G$. That is the set of all jobs required to be finished before starting $j$. We denote this set as $\hat{\mathcal{P}_j}$. Similarly, the \emph{transitive reduction} of $G$ is the minimal set of arcs $\breve{\mathcal{A}}$ such that its transitive closure is equal to the transitive closure of $\mathcal{A}$. We denote by $\breve{\mathcal{P}_j}$ the set of precedences for job $j$ defined by $\breve{\mathcal{A}}$.  Note that in our problem, it is equivalent to use $\mathcal{P}_j$, $\breve{\mathcal{P}}_j$ or $\hat{\mathcal{P}}_j$ as precedence between jobs. Hence, for simplicity and to reduce the size of the problem, for the rest of this work, we assume that the set of precedences of the problem is minimal, that is, $\mathcal{P}_j=\breve{\mathcal{P}_j}$ for all job $j\in\mathcal{J}$.


\subsection{A MIP formulation for the problem}

We start with a \emph{time-indexed} formulation for RCPSP-DC.  Let $x_{jt}\in\{0,1\}$ be a binary variable such that $x_{jt}=1$ indicates that job $j$ is finished at time period $t$. In other words, the job starts at time $t-p_j+1$ and consumes resources during $p_j$ time periods: $[t-p_j+1,\dots ,t]$. Using these binary variables, the problem can be formulated as follows:

\begin{subequations}
\begin{align}
    \max &\ \textrm{NPV}(x) \\
    x_{jt} &= 0 &  t < p_j,  j\in\mathcal{J}\label{eq:Orig1}\\
    \sum_{t=1}^{T} x_{jt} &\leq 1  &  j\in\mathcal{J} \label{eq:OrigOnce}\\
    \sum_{u=1}^t x_{ju} &\leq \sum_{u=1}^{t-p_j} x_{ku} &  \forall k\in\mathcal{P}_j,   j\in\mathcal{J},  t\in\mathcal{T} \label{eq:OrigPrec}\\
    \sum_{j=1}^{N} \sum_{u=1}^{t+p_j-1} q_{jk} \cdot (p_j - (u-t)^+)\cdot  x_{ju} &\leq \sum_{u=1}^t R_{ku} & t\in\mathcal{T} , k\in\mathcal{K} \label{eq:OrigResources}\\
    x_{jt} &\in \{0,1\} &  j\in\mathcal{J},  t\in\mathcal{T} \label{eq:prec_at}
\end{align} \label{eq:Orig}
\end{subequations}

The objective function maximizes the NPV of the schedule, which is given by the expression,
\[ \textrm{NPV}(x) := \sum_{j=1}^{N}\sum_{t=1}^{T}  \frac{f_{j}}{(1+r)^t} x_{jt} \] 

Constraint~\eqref{eq:Orig1} indicates that a job cannot finish before its processing time. Constraint~\eqref{eq:OrigOnce} forces that each job is processed at most one time. Precedence constraint~\eqref{eq:OrigPrec} states that if a job $j$  finishes at time $t$ (or before), then its precedence jobs $k\in\mathcal{P}_j$ must finish at time $t-p_j$ or before. Note that this constraint can also be formulated as $x_{jt} \leq \sum_{u=1}^{t-p_j} x_{ku}$, but \eqref{eq:OrigPrec} provides a stronger formulation for the problem~\citep{lambert2014open}. Finally, constraint~\eqref{eq:OrigResources} provides the resource constraints at each period. In this case, we assume cumulative resources, that is, unused resources at the end of each time period are available for the next period in addition to the new resources received. In Section~\ref{sec:algo}, we discuss the variant where these resources are renewable but not cumulative. Refer to Figure~\ref{fig:explainResourceConstraint} to better understand this equation. The left-hand side of~\eqref{eq:OrigResources} is the total resources of type $k\in\mathcal{K}$ consumed until time $t$, which includes all jobs finished at time $t$ or before ($u\leq t$, in which case a job $j$ consumes  $q_{jk}$ resources during $p_j$ time periods) plus all jobs started before time $t$ and not yet finished ($u>t$, in which case a job $j$ consumes $q_{jk}$ resources during only $p_j - (u-t)$ periods). 

\begin{figure}[tbp]
\begin{center}
    \includegraphics[width=.5\linewidth]{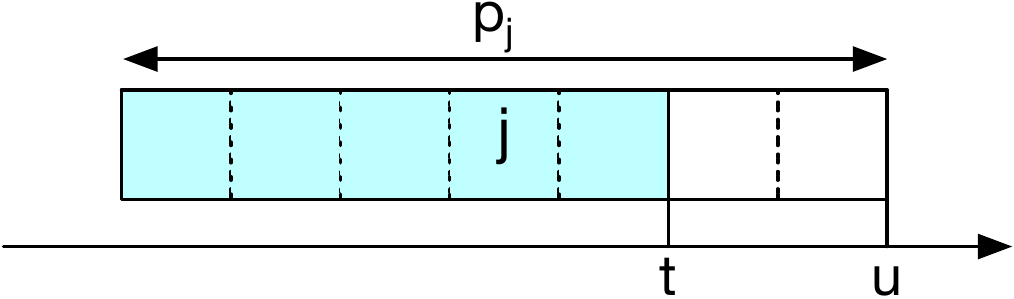}
    \caption{Example of a job $j$ finished at time $u$ ($x_{ju}=1$) which consumes $q_{jk}$ resources of type $k$ along $(p_j-(u-t))$ time periods.}\label{fig:explainResourceConstraint}
\end{center}
\end{figure}

We remark that a usual alternative formulation can be obtained by replacing the interpretation of the decision variable and considering a new binary variable $z_{jt}$ that indicates that a job $j$ has been finished ``by'' time $t$. This reformulation of the problem generates (many) precedence constraints of the form $z_{jt} \leq z_{j't'}$,  plus (a few) resource constraints, which is a structure suitable for decomposition algorithms like the BZ algorithm~\citep{bienstock10}. To improve the clarity of the models, in the following, we consider the original ``at'' formulation presented before, but in~\ref{sec:by}, we describe this alternative formulation.

\section{Geometric time-aggregation} \label{sec:agg}

As previously indicated, we propose a geometric time aggregation for the RCPSP-DC problem. That is, given an \emph{aggregation parameter} $\epsilon$, we re-define the time periods of the problem considering a new set of time \emph{intervals} $I_i = ]\tau_{i-1}, \tau_i]$ where $\tau_i=(1+\epsilon)^i$ for $i\geq 0$ and $I_0=[1]$. For example, if $\epsilon=1$ then $I_0=[1], I_1=]1,2], I_2=]2,4], \ldots, I_s=]2^{s-1},2^s]$.

Let $\mathcal{I}(t)$ be the function that indicates the interval associated to a given time $t\in\mathbb{R}$, that is, $\mathcal{I}(t) = i$ $\Leftrightarrow$ $t\in I_i$. Note that $\mathcal{I}(t)$ can be defined as the function $\mathcal{I}(t):= \lceil\log_{1+\epsilon}(t)\rceil$. We also assume that $\mathcal{I}(t) = 0$ for $t<1$. Finally, we denote by $T_I=\mathcal{I}(T)$  the number of intervals required induced by $T$, and by $\mathcal{T}_I = \{1,\ldots, T_I\}$ the set of intervals of the problem.

Using these newly defined time intervals, we formulate a MIP problem similar to~\eqref{eq:Orig}. Similarly, we define a binary variable $X_{js}\in\{0,1\}$ such that $X_{js}=1$ if job $j$ is finished during interval $I_s$.  However, we need to provide some assumptions on how to interpret this schedule.

One of the difficulties in writing this problem is the non-uniform time intervals affecting the precedences and resource constraints. Since time intervals are no longer uniformly sized, the number of intervals to consider when rewriting these constraints will depend on the current interval. For example, on the original formulation, if $k\in \mathcal{P}_i$ and we schedule job $i$ to finish at time $t$, then job $k$ must finish at time $t-p_i$ or before; this rule does not depend on the specific value of $t$. With non-uniform time intervals, this is no longer true, and the precedence between jobs $j$ and $k$ can consider several time intervals between them, or they can even be scheduled on the same interval. 

For the resource constraints, we interpret that if $X_{js}=1$, job $j$ is finished at time $\tau^s$. Hence, an equivalent formulation of constraint~\eqref{eq:OrigResources} can be written as,
\begin{equation*}
        \sum_{j=1}^{N} \sum_{u=1}^{\mathcal{I}(\tau_t+p_j)-1} q_{jk} \cdot (p_j - (\tau_u-\tau_t)^+)\cdot  X_{ju} \leq \sum_{s=1}^{\lceil\tau_t\rceil}R_{ks} \qquad  t\in\mathcal{T}_I , k\in\mathcal{K}. 
\end{equation*}

\begin{figure}[tbp]
\begin{center}
    \includegraphics[width=.5\linewidth]{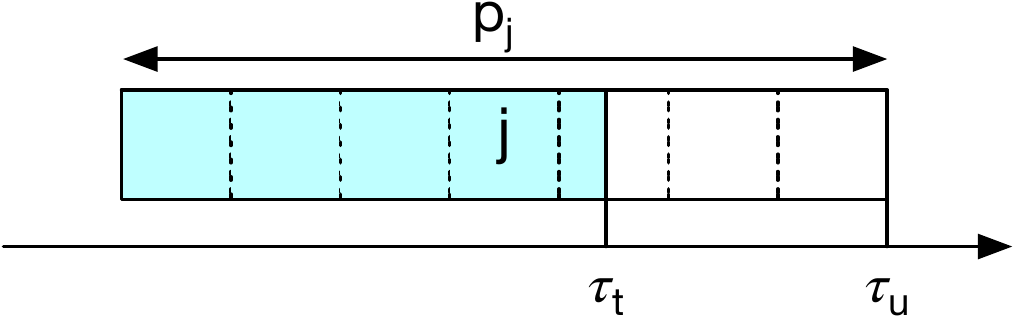}
    \caption{Equivalent to Figure~\ref{fig:explainResourceConstraint} for geometric time intervals}\label{fig:explainResourceConstraintAgg}
\end{center}
\end{figure}

The interpretation of these constraints is similar to before. The consumed resources at time $\tau_t$ consider the jobs finished before that time, and the jobs started before $\tau_t$ and not yet finished (see Figure~\ref{fig:explainResourceConstraintAgg}). For the right-hand side of the constraint, we assume that the cumulative available resources at the end of interval $I_t$ are given by the resources available at time $s\leq \lceil\tau_t\rceil$.

For the precedence constraints, we can assume that if a job $j$ is scheduled to finish at interval $s$, then its precedence jobs $k\in\mathcal{P}_j$ must finish at time $\mathcal{I}(\tau_s-p_j)$ or before. Hence, the number of intervals required between $j$ and $k$ will depend on the scheduled interval of job $j$. Moreover, for longer intervals (where $\tau_s-\tau_{s-1}>p_j$), jobs $k$ and $s$ could be scheduled in the same interval without violating its precedence constraint. However, this shows that we can no longer rely on the transitivity of the precedence relationship to formulate these constraints. In fact, in the latter example, suppose that a job~$l$ is in $\mathcal{P}_k$ and $p_k < \tau_s-\tau_{s-1}$ but $p_j+p_k > \tau_s-\tau_{s-1}$. In this case, the precedences between $j$ and $k$ and between $k$ and $l$ individually indicates that both pairs of jobs can finish on the same interval $s$, but this is not possible for all three because $p_j+p_k > \tau_s-\tau_{s-1}$.

To deal with this issue, we use the graph representation of the precedence constraints. Let $G=(\mathcal{J},\mathcal{A})$ be the graph of precedences between jobs, and let associate a length $w_{jk}:=p_j$ to each arc $jk\in \mathcal{A}$. For each job $j\in\mathcal{J}$ and for each job $i$ in the transitive closure of node $j$ (that is, a reachable node from $j$), we define $\Delta_{ji}$ as the length of the longest path from $j$ to $i$ over $G$. In other words, $\Delta_{ji}$ is the minimum time span induced by the original precedences between the finishing times of jobs $i$ and $j$. Using these coefficients, we can rewrite the analogous precedence constraints~\eqref{eq:OrigPrec} for geometric intervals as
\[ 
    \sum_{u=1}^s X_{j,u} \leq \sum_{u=1}^{\mathcal{I}(\tau_{s}-\Delta_{jk})} X_{k,u} \qquad  k\in\hat{\mathcal{P}}_j, j\in\mathcal{J}, s\in\mathcal{T}_I.
\]    
Note that if $\tau_{s} < \Delta_{jk}$ then $X_{j,u}=0$ for all $u\leq s$.
    
Finally, for the objective function, we assume that for jobs with positive profit ($f_i>0$) that are scheduled, the resulting profit is obtained \underline{at the beginning} of the interval. For jobs with a negative coefficient ($f_i<0$), this loss is obtained \underline{at the end} of the interval. This modeling decision is intended to obtain an upper bound for the original problem, but equivalent results can be obtained under other assumptions. Hence, we denote the objective function of a given schedule $X$ as 
\[ \widehat{\textrm{NPV}}(X) := \sum_{\substack{j=1\\f_j>0}}^{N}\sum_{s=1}^{T_I}  \frac{f_{j}}{(1+r)^{\tau_{s-1}}}  X_{js} + \sum_{\substack{j=1\\f_j<0}}^{N}\sum_{s=1}^{T_I}  \frac{f_{j}}{(1+r)^{\tau_{s}}}  X_{js}.\]

In summary, the resulting MIP formulation for the problem with geometric time intervals is given by
\begin{subequations}
\begin{align}
\max &\  \widehat{\textrm{NPV}}(X),
\end{align}
subject to the following constraints,
\begin{align}
    X_{jt} &= 0 &  t < \mathcal{I}(p_j),  j\in\mathcal{J},\label{eq:Aggelim}\\
    \sum_{t=1}^{T_I} X_{jt} &\leq 1  &  j\in\mathcal{J},\\
    \sum_{u=1}^t X_{ju} &\leq \sum_{u=1}^{\mathcal{I}(\tau_{t}-\Delta_{jk})} X_{ku} & k\in\hat{\mathcal{P}}_j, j\in\mathcal{J}, t\in\mathcal{T}_I, \label{eq:Aggprec_at}\\
    \sum_{j=1}^{N} \sum_{u=1}^{\mathcal{I}(\tau_t+p_j)-1} q_{jk}  (p_j -& (\tau_u-\tau_t)^+)  X_{ju} \leq  \sum_{s=1}^{\lceil\tau_t\rceil}R_{ks} & t\in\mathcal{T}_I , k\in\mathcal{K},\label{eq:AggResources}\\
    X_{jt} &\in \{0,1\} &  j\in\mathcal{J},  t\in\mathcal{T}_I.
\end{align} \label{eq:Agg}
\end{subequations}

Note that this formulation can also be reformulated using the ``by`` variables (as in~\ref{sec:by}), obtaining the required structure to use decomposition algorithms like the BZ algorithm.


The following proposition shows that any feasible solution of Problem~\eqref{eq:Orig} can be transformed into a feasible solution of Problem~\eqref{eq:Agg}, with a higher objective value.

\begin{proposition}\label{prop:UB}
Let $x^*$ be a feasible solution for Problem~\eqref{eq:Orig}, and let be $X^*\in\{0,1\}^{\mathcal{J}\times \mathcal{T}_I}$ such that if $x^*_{j,t}=1$ then $X^*_{j,\mathcal{I}(t)}=1$ for each $j\in\mathcal{J}$, and $X^*_{js}=0$ for $s\in \mathcal{T}_I$, $s\neq \mathcal{I}(t)$. Therefore, $X^*$ is a feasible solution for Problem~\eqref{eq:Agg}, and $\textrm{NPV}(x^*) \leq \widehat{\textrm{NPV}}(X^*)$.
\end{proposition}

\begin{proof}
We focus on precedence and resource constraints to prove the feasibility of $X^*$ because all other constraints are trivially satisfied. 

Let $j\in\mathcal{J}$ and let $k\in\hat{\mathcal{P}}_j$. By definition of $\Delta_{jk}$, if $x^*_{jt}=1$ then $x^*_{ks}=1$ for some $s\leq t-\Delta_{jk}$. Therefore,  $X^*_{j,\mathcal{I}(t)}=1$ and $X^*_{k,s'}=1$ for some  $s'\leq \mathcal{I}(t-\Delta_{jk})$. Since $t\leq\tau_{\mathcal{I}(t)}$ then precedence constraints~\eqref{eq:Aggprec_at} are satisfied. 

For the resource constraints, let $s\in\mathcal{T}_I$ and let $t'=\lceil\tau_s\rceil$. We first assume that $t'>\tau_s$. 
Since $x^*$ is feasible for Problem~\eqref{eq:Orig} it satisfies
\[ \sum_{j=1}^{N} \sum_{u=1}^{t'-1} q_{jk} \cdot p_j \cdot  x^*_{ju} + \sum_{u=t'}^{t'+p_j-1} q_{jk} \cdot (p_j - (u-t')^+)\cdot  x^*_{ju} \leq  \sum_{u=1}^{t'} R_{ku}.\]

On the one hand, note that 
\begin{align*}
\sum_{j=1}^{N} \sum_{u=1}^{t'-1} q_{jk} \cdot p_j \cdot  x^*_{ju} &= \sum_{j=1}^{N} \sum_{u=1}^{s} q_{jk} \cdot p_j \cdot  X^*_{ju}, 
 \end{align*}
because if $X^*_{ju}=1$ for $u\leq s$ then $x^*_{ju}=1$ for $u\leq t'-1 \leq\tau_s < t'$.

On the other hand, 

\begin{align*}
    \sum_{u=t'}^{t'+p_j-1} q_{jk} \cdot (p_j - (u-t')^+)\cdot  x^*_{ju}  
    &\geq \sum_{u=t'}^{t'+p_j-1} q_{jk} \cdot (p_j - (\tau_{\mathcal{I}(u)}-\tau_s)^+)\cdot  x^*_{ju}, \\
    &= \sum_{u=\mathcal{I}(t')}^{\mathcal{I}(t'+p_j-1)} q_{jk} \cdot (p_j - (\tau_{u}-\tau_s)^+)\cdot  X^*_{ju}, \\
    &\geq \sum_{u=s+1}^{\mathcal{I}(\tau_s+p_j)-1} q_{jk} \cdot (p_j - (\tau_{u}-\tau_s)^+)\cdot  X^*_{ju}.
\end{align*}

The first inequality is obtained because $u\leq \tau_{\mathcal{I}(u)}$ and $t'>\tau_s$. For the last inequality, the lower limit of the sum can be replaced because if $\mathcal{I}(t')>s+1$ then $X^*_{ju}=0$ for $s+1\leq u < \mathcal{I}(t')-1$ for any job $j\in\mathcal{J}$. Also, the upper limit of the sum can be replaced because if $\mathcal{I}(t'+p_j-1)>\mathcal{I}(\tau_s+p_j)-1$ and since $\mathcal{I}(t'+p_j-1)\leq \mathcal{I}(\tau_s+p_j)$ then  $\mathcal{I}(t'+p_j-1)= \mathcal{I}(\tau_s+p_j)$. Hence job~$j$ will be finished at the end of the interval~$\mathcal{I}(\tau_s+p_j)$, not consuming resources on the interval $s$.

In the remaining case in which $\tau_s\in\mathcal{T}$ (so $t'=\tau_s$), the result can be obtained by the same arguments but considering the expression 
 \[ \sum_{j=1}^{N} \sum_{u=1}^{t'} q_{jk} \cdot p_j \cdot  x^*_{ju} + \sum_{u=t'+1}^{t'+p_j-1} q_{jk} \cdot (p_j - (u-t')^+)\cdot  x^*_{ju} \leq  \sum_{u=1}^{t'} R_{ku}. \]
Hence, constraints~\eqref{eq:AggResources} are satisfied by $X^*$. 

Finally, to prove the bound on the objective function, note that if $x^*_{jt}=1$ then the contribution to the objective value of Problem~\eqref{eq:Orig} is $f_i/(1+r)^{t}$, which is smaller than $f_i/(1+r)^{\tau_{\mathcal{I}(t)}}$ if $f_i<0$ and smaller than $f_i/(1+r)^{\tau_{\mathcal{I}(t)-1}}$ if $f_i>0$, so $\textrm{NPV}(x^*) \leq \widehat{\textrm{NPV}}(X^*)$.
\end{proof}

\section{An approximation algorithm for the RCPSP-DC} \label{sec:algo}

Problem~\eqref{eq:Agg} provides a geometric time-aggregated version of Problem~\eqref{eq:Orig}, with a reduced number of time intervals. Due to this, it is more appropriate to solve using standard MIP solvers. However, the resulting optimal solution needs to be transformed back to the domain of the original problem. This transformation can be easily done by applying a topological sorting algorithm over the graph of precedences and assigning the jobs to the earliest available time period of $\mathcal{T}$ provided by the solution of the aggregated problem. More precisely, a topological order of the jobs $\sigma(\mathcal{J})$ is an order of the jobs $\mathcal{J}$ such that each job has all its precedences before itself in the ordering. With this order, for each interval~$s$, we sequentially assign jobs $j$ with $X^*_{js}=1$ in the topological order to the earliest time $t>\tau_{s-1}$ with enough available resources to complete the job. See Algorithm~\ref{alg:sort} for more details.

\begin{algorithm}
 \caption{Interval Aggregation Approximation Algorithm }\label{alg:sort}
 \begin{algorithmic}
 \Require Aggregation parameter $\epsilon$
 \Ensure Finishing time $C_j$ for each job $j\in\mathcal{J}$
 \State $X^* \gets$ Optimal solution of Problem~\eqref{eq:Agg}  with parameter $\epsilon$
 \State $\sigma(\mathcal{J})\gets$ a topological order of $\mathcal{J}$ according to $\mathcal{P}$
 \ForAll{$s\in\mathcal{T}_I$}
    \ForAll{$j\in\sigma(\mathcal{J})$ such that $X^*_{j,s}=1$}
        \State $t \gets \max\{\lfloor\tau_{s-1}\rfloor+1, C_k+p_j \ \forall k\in\mathcal{P}_j\}$ 
        \While{not enough resources to assign job $j$ to $[t-p_j, t]$}
            \State $t \gets t+1$
        \EndWhile
        \State $C_j \gets t$   
    \EndFor
 \EndFor
 \end{algorithmic}
\end{algorithm}

It is easy to see that Algorithm~\ref{alg:sort} provides a feasible solution for Problem~\eqref{eq:Orig}, and by Proposition~\ref{prop:UB} its objective value is smaller than $\widehat{\textrm{NPV}}(X^*)$. However, this objective value could potentially be far away from the true optimal value of Problem~\eqref{eq:Orig}.  

In the following theorem, we provide an approximation bound on the quality of the obtained solution for the case that all profits are non-negative.
 
\begin{theorem}\label{thm:main} If $f_j\geq 0$ for all $j\in\mathcal{J}$, then Algorithm~\ref{alg:sort} is a $\gamma$-approximation algorithm for Problem~\eqref{eq:Orig}, where
\[\gamma := (1+r)^{-T\cdot \frac{2\epsilon}{1+\epsilon}}\]
\end{theorem}
 
\begin{proof}
Let $\hat{X}^*$ be the optimal solution of Problem~\eqref{eq:Agg},
let $x^*$ be the optimal solution of Problem~\eqref{eq:Orig}, and let $X^{[x^*]}$ be the solution constructed from $x^*$ as in Proposition~\ref{prop:UB}.
 
By  Proposition~\ref{prop:UB}, we now that 
\[ \text{OPT} := \text{NPV}(x^*) \leq  \widehat{\textrm{NPV}}(X^{[x^*]}) \leq  \widehat{\textrm{NPV}}(\hat{X}^*). \]
  
Let $\hat{x}$ be the solution obtained from Algorithm~\ref{alg:sort}, and define
\[ \text{ALG} := \text{NPV}(\hat{x}) = \sum_{j=1}^{N}\sum_{t=1}^{T}  \frac{f_{j}}{(1+r)^t} \hat{x}_{jt}. \] 
 

Note that for each job $j\in\mathcal{J}$, if $\hat{X}^*_{js}=1$ then Algorithm~\ref{alg:sort} assigns to $j$ a finishing time in a period $C_j$ which is greater than $\tau_{s-1}$. 
On the other hand, if $\hat{X}^*_{js}=1$ then $C_j$ cannot be greater than $\tau_{s+1}$. If $\hat{X}^*_{js}=1$, then there are enough available resources to schedule job $j$ and all its precedence jobs at the end of the interval $I_s$, that is, time $\tau_s$. However, this cannot ensure that $C_j\leq \tau_s$. Because when transforming the solution $\hat{X}^*$ to times periods, the available resources at times at the beginning of interval $I_s$ could not be enough to schedule job $j$ and all its precedent jobs. However, this does not occur on the next interval $I_{s+1}$ because we know that there are enough resources at the beginning of this interval, which is $\tau_{i}$. So, in the worst case, job $j$ will consume the resources of interval $I_{s+1}$ starting from time $\tau_s$, so it will finish at a time $C_j\leq \tau_s + (\tau_s-\tau_{s-1}) = \tau_s \cdot  \tfrac{1+2\epsilon}{1+\epsilon}$. 

Therefore, since $f_j\geq 0$ for all $j\in\mathcal{J}$, then
\[
\sum_{j=1}^{N}\sum_{s=1}^{T_I}  \frac{f_{j}}{(1+r)^{\tau_s \cdot  \tfrac{1+2\epsilon}{1+\epsilon}}}  \hat{X}^*_{js} \leq 
\sum_{j=1}^{N}\sum_{t=1}^{T}  \frac{f_{j}}{(1+r)^t} \hat{x}_{jt} \leq  
\sum_{j=1}^{N}\sum_{s=1}^{T_I}  \frac{f_{j}}{(1+r)^{\tau_{s-1}}}  \hat{X}^*_{js}.
\]
So, 
\[ \text{ALG} \geq \sum_{j=1}^{N} \frac{(1+r)^{\tau_{s-1}}}{(1+r)^{\tau_s \cdot  \tfrac{1+2\epsilon}{1+\epsilon}}} \sum_{s=1}^{T_I}\frac{f_{j}}{(1+r)^{\tau_{s-1}}}  \hat{X}^*_{js} \geq \gamma \cdot   \widehat{\textrm{NPV}}(\hat{X}^*),
\]
where
\begin{equation*}
0 < \gamma \leq \min_{s\in\mathcal{T}_I} \left\{\frac{(1+r)^{\tau_{s-1}}}{(1+r)^{\tau_s \cdot  \tfrac{1+2\epsilon}{1+\epsilon}}}\right\}.  \label{eq:boundGamma}
\end{equation*}

Combining the previous expression, we get that 
\[ \gamma \cdot   \widehat{\textrm{NPV}}(\hat{X}^*) \leq \text{ALG} \leq \text{OPT} \leq \widehat{\textrm{NPV}}(\hat{X}^*) \leq \frac{1}{\gamma} \text{ALG},\]
proving that $\text{ALG} \geq \gamma \cdot \text{OPT}$. 

Finally, to compute the best value for $\gamma$, note that 
\[ \gamma \leq \min_s \  (1+r)^{\tau_s \cdot \left(\frac{1}{1+\epsilon} - \frac{1+2\epsilon}{1+\epsilon}\right)}, \]
and this minimum is achieved in the last time period, hence
\[ \gamma \leq (1+r)^{-T\cdot\frac{2\epsilon}{1+\epsilon}}, \]
concluding the proof.
\end{proof}
 
 \begin{remark}
The condition of $f_j\geq 0$ for all jobs $j\in\mathcal{J}$ is required to obtain any approximation factor. In fact, consider a problem with only two jobs $j$ and $j'$, with processing times $p_j$ and $p_{j'}$, where $j$ is a precedence of $j'$. Assume that $f_j<0$ and $f_{j'}=-f_j + \delta$, with a $\delta>0$ small enough such that
\[ \frac{f_{j}}{(1+r)^{p_{j}}}  < - \left( \frac{-f_{j}+\delta}{(1+r)^{p_{j}+p_{j'}}}\right). \]
In this case, the optimal solution to the problem is not to schedule any job, obtaining an objective value of 0. However, for any aggregation parameter~$\epsilon$, it is possible to choose a value for the processing times such that $p_{j}$ and $p_{j}+p_{j'}$ belongs to the same interval $s:=\mathcal{I}(p_j) = \mathcal{I}(p_j+p_{j'})$. Therefore, the optimal solution $X_{js}=X_{j's}=1$  for the aggregated problem has an objective value 
\[\widehat{NPV}(X) =  \frac{f_{j}}{(1+r)^{\tau_{s}}}  + \frac{-f_{j}+\delta}{(1+r)^{\tau_{s-1}}} >\frac{\delta}{(1+r)^{\tau_{s}}} > 0. \]
Hence, the approximation algorithm will schedule both jobs, obtaining a negative objective value for the original problem so that no bounded approximation factor can be obtained for this case. 
Note that this counterexample also applies if we modify the definition of $\widehat{NPV}(X)$ to assign the profits of all jobs to the beginning or the end of the time intervals. Nevertheless, in practice, the quality of the solution provided by the approximation algorithm is good enough even with negative profits, as we show in Section~\ref{sec:comput}.
\end{remark}

To evaluate the quality of this approximation factor, we can rely on the usual discount rates for project evaluation process. The parameter $r$ is the discount rate, which generally goes from $5\%$ to $10\%$ (annual rate), depending on the industry. For a project with a time horizon going from 1 to 5 years, we can evaluate the approximation factor $\gamma$. Note that $\gamma$ does not depend on the granularity of the time periods. If we consider shorter periods (for example, days), we need to use the corresponding daily discount rate $r_{day}=(1+r)^{1/365}-1$, so $(1+r)^T$ will have the same value independently of the length of the individual time periods.

\begin{figure}[htbp]
    \centering
    \includegraphics[width=.48\linewidth]{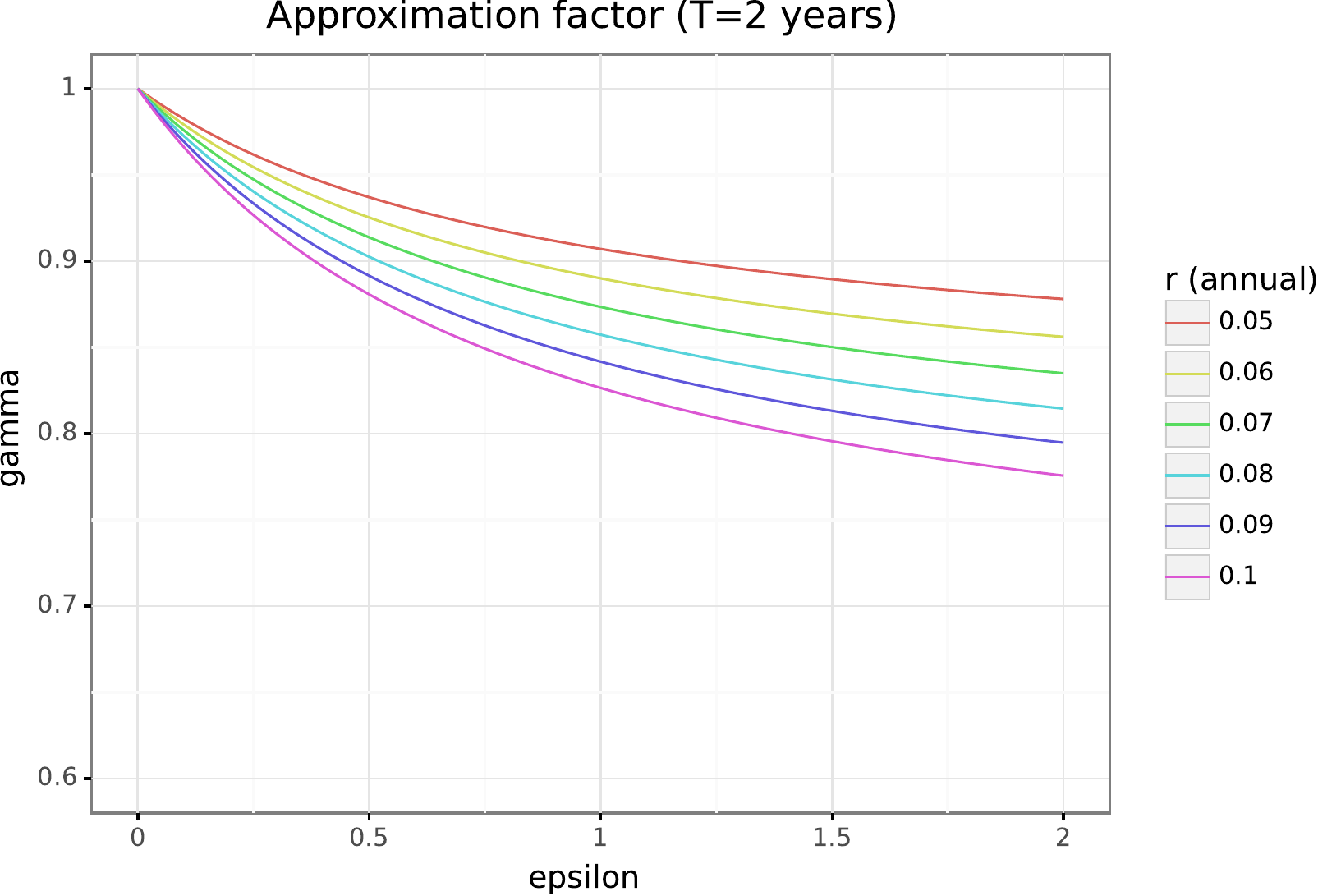} 
    \includegraphics[width=.48\linewidth]{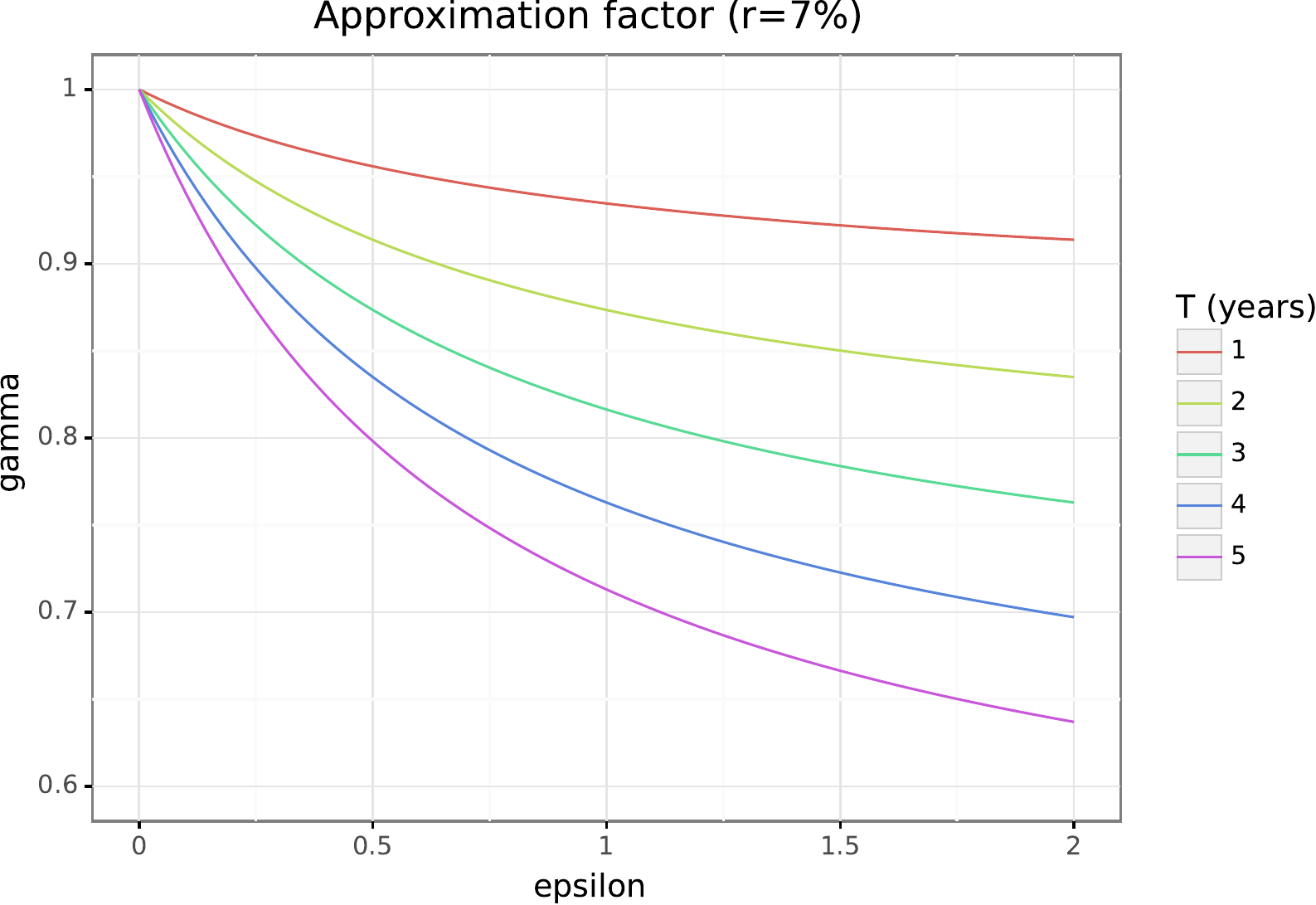}
    \caption{Evaluation of the obtained approximation factor for different time horizons and discount factors}
    \label{fig:gamma}
\end{figure}
 
Figure~\ref{fig:gamma} shows the value of $\gamma$ for different discount rates and horizons of the problem. It can be seen that for a two-year project, even with high discount rates, the algorithm is providing solutions that have an objective value higher than 80\% the value of the optimal solution. This approximation factor is more sensitive to the time horizon of the problem. Nevertheless, a time aggregation with $\epsilon=1$ ensures a 70\% of the optimal value for projects of up to 5 years while reducing the number of periods of the aggregated problem by a logarithmic factor of base 2. 
However, as we will see in Section~\ref{sec:comput}, these theoretical values are a lower bound, and in practice, much better approximations are obtained.

We finish this section by discussing how to extend these ideas to the case of renewable non-cumulative resource constraints. In this classical variant of the RCPSP, the resources available at each time period are not cumulative, so the unused resources at a given period are discarded. In this case, a MIP formulation similar to Problem~\eqref{eq:Orig} can be obtained  by replacing the resource constraints~\eqref{eq:OrigResources} to
\begin{equation}
 \sum_{j=1}^{N} \sum_{u=t}^{t+p_j-1} q_{jk} x_{ju} \leq R_{kt} \qquad  t\in\mathcal{T} , k\in\mathcal{K}. \label{eq:OrigResourcesRenew}
\end{equation}

That is, we consider a resource constraint for each time period $t$ and resource $k$, including only the jobs being processed during time $t$. We consider the total resource consumption during the given interval to construct an equivalent formulation for the time-aggregated Problem~\eqref{eq:Agg}. As before, assuming that $X_{js}=1$ if job $j$ finishes at the end of the interval $I_s$, then we write this constraint as
\[ \sum_{j=1}^{N} \sum_{u=t}^{\mathcal{I}(\tau_t+p_j)-1} q_{jk} \cdot \min\{p_j - (\tau_u-\tau_t),\tau_t-\tau_{t-1}\}\cdot  X_{ju} \leq  \sum_{s=\lceil\tau_{t-1}\rceil+1}^{\lceil\tau_t\rceil}R_{ks}, \]
for each $t\in\mathcal{T}_I$ and $k\in\mathcal{K}$. The left-hand side considers all jobs being processed at interval $u$. For a job $j$ finishing at interval $t$, its resource consumption is $q_{ik} p_j$ if $p_j$ is smaller than the length of the interval, or $q_{jk}(\tau_t-\tau_{t-1})$ if not. Similarly, for jobs finishing at time intervals $u>t$, we consider the fraction of the processing time at time $\tau_u$, which is given by $p_j-(\tau_u-\tau_t)$. For the right-hand side, we consider the sum of the available resources at each particular time period belonging to the time interval but round up its extremes $\tau_t$ and $\tau_{t-1}$. 

In this way, Proposition~\ref{prop:UB} is also true for the case of non-cumulative resources. However, Theorem~\ref{thm:main} cannot be extended to this case. For example, consider the case of a job $j$ requiring $q_{jk}$ resources with $q_{jk}>R_{kt}$ for all $t\in\mathcal{T}$. Hence, this job cannot be scheduled at any time due to constraint~\eqref{eq:OrigResourcesRenew}. Still, the aggregated problem will be able to include this job in a sufficiently long time interval, making it impossible to achieve an approximation guarantee for Algorithm~\ref{alg:sort}. Nevertheless, it is expected that the optimal solution to the problem schedules most jobs as soon as possible, so in practice, the active constraints will be mostly precedence constraints (in which case the resource constraints do not provide an important role) or mostly resource constraints (in which case, consuming all available resources at each time period make the problem equivalent to cumulative resources); so the approximation algorithm will still provide a good solution, as we show in the next section.

\section{Computational experiments} \label{sec:comput}

\subsection{Implementation details}

All optimization models and algorithms were coded on Python v3.9 and solved using Gurobi v9.0.2, with its default parameters. We formulate the MIP model as in Problem~\ref{eq:Agg}, except for the precedence constraints, relying on the solver to eliminate unnecessary variables and reduce the problem's size. 

For the precedence constraints, to include all precedences in $\hat{\mathcal{P}}$ requires $\mathcal{O}(|\mathcal{T}_I|\cdot |\mathcal{J}|^2)$ constraints, which make the problem unnecessarily large. To avoid this problem, we compute the transitive reduction of the precedence graph for each interval. More precisely, we first compute the transitive closure $\hat{\mathcal{P}}$ and compute the longest path between each pair of nodes $\Delta_{ij}$ where a path from $i$ to $j$ in $\hat{\mathcal{P}}$ exists. Then, to construct the precedences constraints for interval $t\in\mathcal{T}_I$, for each job $i$ and for each $j\in\hat{\mathcal{P}}_i$ we compute the interval $\mathcal{I}(\tau_t - \Delta_{ij})$ which is the last interval where job $j$ must be scheduled to be able to schedule job $i$ at time interval $t$. If $\tau_t < \Delta_{ij}$ for some $j\in\hat{\mathcal{P}}_i$, then it means that job $i$ cannot be scheduled in this interval, so we impose that $X_{i,t}=0$ in constraint~\eqref{eq:Aggelim}. In the other case, we assign a \emph{weight} to the arc $(i,j)$ in $\hat{\mathcal{P}}$ equal to $w_{ij}=t-\mathcal{I}(\tau_t - \Delta_{ij})$. Finally, we compute a modified transitive reduction of this graph for removing redundant arcs. That is, recursively, we remove an arc $(i,j)$ if there exists another node $k$ such that $(i,k)$ and $(k,j)$ are in $\hat{\mathcal{P}}$ and $w_{ij} < w_{ik}+w_{kj}$. In other words, we remove the precedence between $i$ and $j$ for time interval $t$ because it is already implied by the precedence between $i$ and $k$ and between $k$ and $j$ for the same interval. This procedure considerably reduces the number of precedence constraints to include in the MIP problem. 

For Algorithm~\ref{alg:sort}, we can use any topological order $\sigma(\mathcal{J})$ to assign the jobs. However, there are many topological orders, each potentially providing a different solution. To obtain better solutions, we prioritize the jobs with higher value $f_i$ in the topological order so that these jobs will be scheduled at earlier time periods. This approach can be implemented using a \emph{min heap queue}, using as the priority of each element the pair $(C^I_j, -f_i)$, where $C^I_j$ is the interval assigned by the optimal solution of the aggregated problem. Initially, we insert all jobs scheduled by the aggregated problem without any precedence into the queue. So, the queue will prioritize the jobs available to be scheduled in the order of $C^I_j$ and with the best profit among all jobs with the same $C^I_j$. Finally, each time a job $i$ is removed from the queue, we revise the jobs with $i$ as precedence to insert them into the queue if all its precedences have already been scheduled.

\subsection{Results for the PSPLib Instances}

The first two datasets come from the classical instances of PSPlib~\citep{kolisch1997psplib}. These datasets contain single-mode RCPSP problems with 30 jobs (dataset J30) with $119\sim 210$ time periods and 120 jobs (dataset J120) with $551\sim 744$ time periods, with a total of 480 and 600 different instances, respectively. All instances have four types of resources. Since these instances do not have a profit for completing a job, we assume $f_j=1$ for all jobs, and a discount rate of $0.1\%$ per period (that is, assuming that a time period is a week, then the equivalent annual discount rate is  $\sim 5.4\%$ and the time horizon is $2\sim 14$ years.)
 
To evaluate the quality of the solutions generated by Algorithm~\ref{alg:sort}, we compute a gap of optimality for a given solution $x$ by comparing its objective value with the upper bound provided by the optimal solution $X^*$ of the aggregated Problem~\ref{eq:Agg} solved to obtain $x$. 
\begin{equation}
    \text{\textsc{Gap}}(x) = \frac{\widehat{\text{NPV}}(X^*) - \text{NPV}(x) }{\text{NPV}(x)}. \label{eq:defGap}
\end{equation}

Table~\ref{tab:res_Psp} shows a summary of results for the J30 and J120 datasets. We solve both versions of resource constraints for these datasets, with cumulative and non-cumulative resources. Columns show the number of instances solved up to optimality for a time limit of 24 hours, the average solving time, and the average solution gap as defined in~\eqref{eq:defGap}. For the solving time, we use only the solving time reported by Gurobi to solve the aggregated MIP model.

\begin{table}[htbp]
    \centering
    \footnotesize
    \setlength{\tabcolsep}{4pt}
    \begin{tabular}{c|c||c|c|c||c|c|c}
   \multicolumn{2}{c||}{} & \multicolumn{3}{c||}{Cumulative resources} & \multicolumn{3}{c}{Non-cumulative resources}\\ \hline
    &  & Instances & Avg. Time & Avg. Gap & Instances & Avg. Time & Avg. Gap \\
   Data & $\epsilon$ & Solved & (sec) & (\%) &  Solved & (sec) & (\%) \\ \hline 
\multirow{4}{*}{J30}
 & 0.1 & 480 & 3.698 & 0.20 & 480 & 1.238 & 0.77\\
 & 0.2 & 480 & 0.161 & 0.34 & 480 & 0.079 & 0.94\\
 & 0.5 & 480 & 0.014 & 0.60 & 480 & 0.008 & 1.19\\
 & 1.0 & 480 & 0.004 & 0.91 & 480 & 0.004 & 1.50\\ \hline
\multirow{4}{*}{J120}
 & 0.1 & 448 & 2479.6 & 0.53 & 568 & 1920.2 & 2.81\\
 & 0.2 & 541 & 1780.1 & 1.21 & 600 &  0.967 & 3.35\\
 & 0.5 & 600 & 280.1  & 2.01 & 600 &  0.079 & 3.75\\
 & 1.0 & 600 & 0.078  & 2.81 & 600 &  0.024 & 4.19\\ 
 \end{tabular}
    \caption{Summary of results for PSPlib instances}
    \label{tab:res_Psp}
\end{table}

It can be seen that increasing the aggregation factor $\epsilon$ considerably reduces the solving time without compromising the quality of the solution. For example, for the J30 dataset, the aggregated problem can be solved in a fraction of a second, with average gaps of less than 1\%. This improvement is more notorious for the larger problems in the J120 dataset. While a value of $\epsilon=0.1$ cannot solve all instances in the 24-hour time limit, larger values of $\epsilon$ considerable reduce the required time, being able to solve all instances in less than a second for $\epsilon=1$, with an average optimality gap of $2.8\%$. Note that for non-cumulative resource constraints, the results are similar, with lower computation times but slightly more significant gaps. 

\begin{figure}[tbp]
    \centering
    \includegraphics[width=.95\linewidth]{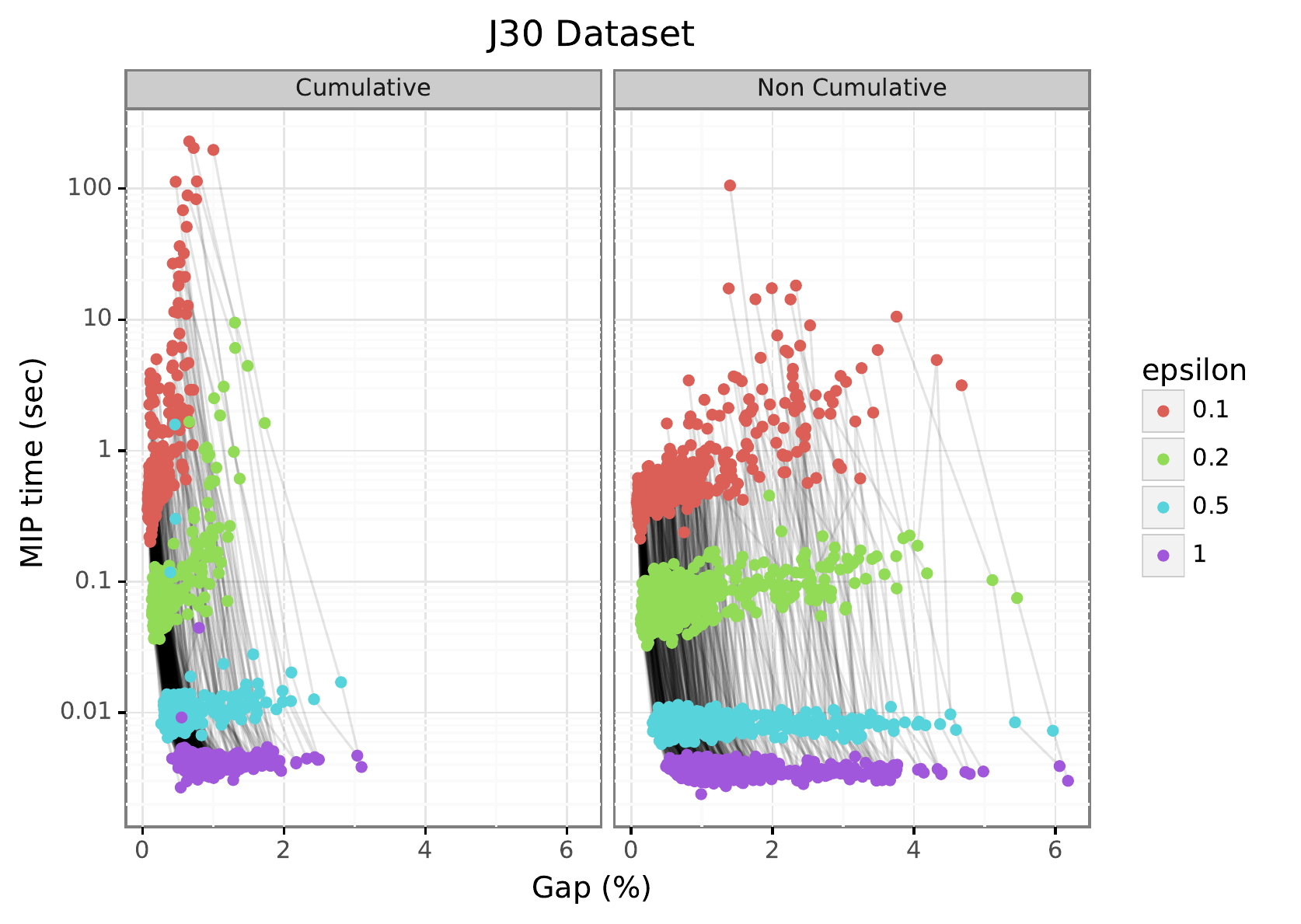}
    \includegraphics[width=.95\linewidth]{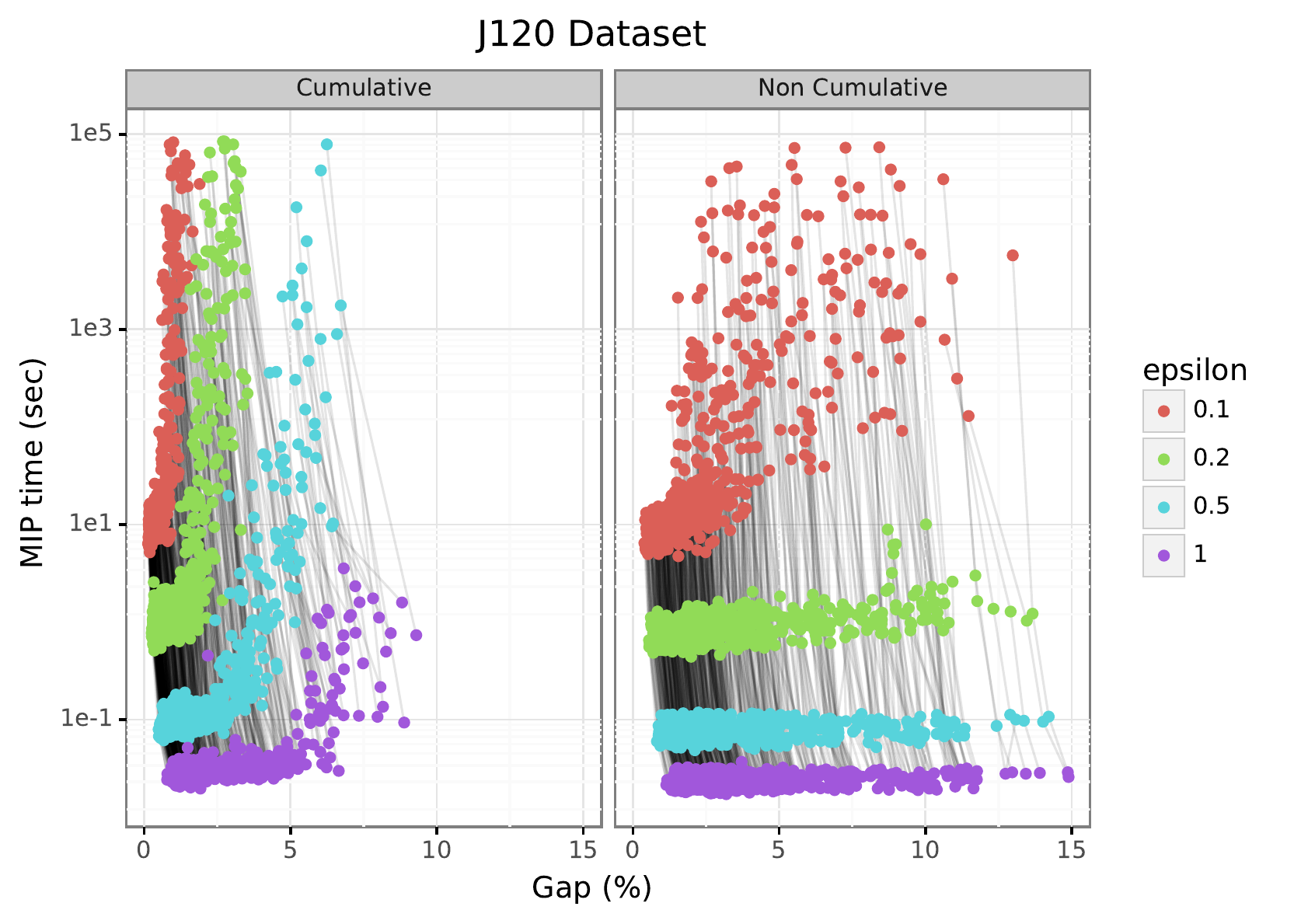}
    \caption{Solving time and resulting optimality gap for each instance with different aggregation parameters}
    \label{fig:res_Psp}
\end{figure}

Figure~\ref{fig:res_Psp} shows a more detailed analysis of each instance in the datasets. The figure shows the resulting gap (x-axis) and computation time (y-axis, in log-scale) of each instance for the different aggregation parameters $\epsilon$, with a thin line connecting the values for the same instances. It can be seen that instances with the most significant gaps for small values of $\epsilon$ remain with the most important gaps for larger values of $\epsilon$, showing consistency between the models. Also, the figure shows a positive correlation between the obtained gap and the solving time, particularly for smaller $\epsilon$. Surprisingly, this correlation is lower for the problem with non-cumulative resource constraints.

\subsection{Results for underground mining instances}

\begin{figure}[tbp]
    \centering
    \includegraphics[width=.8\linewidth]{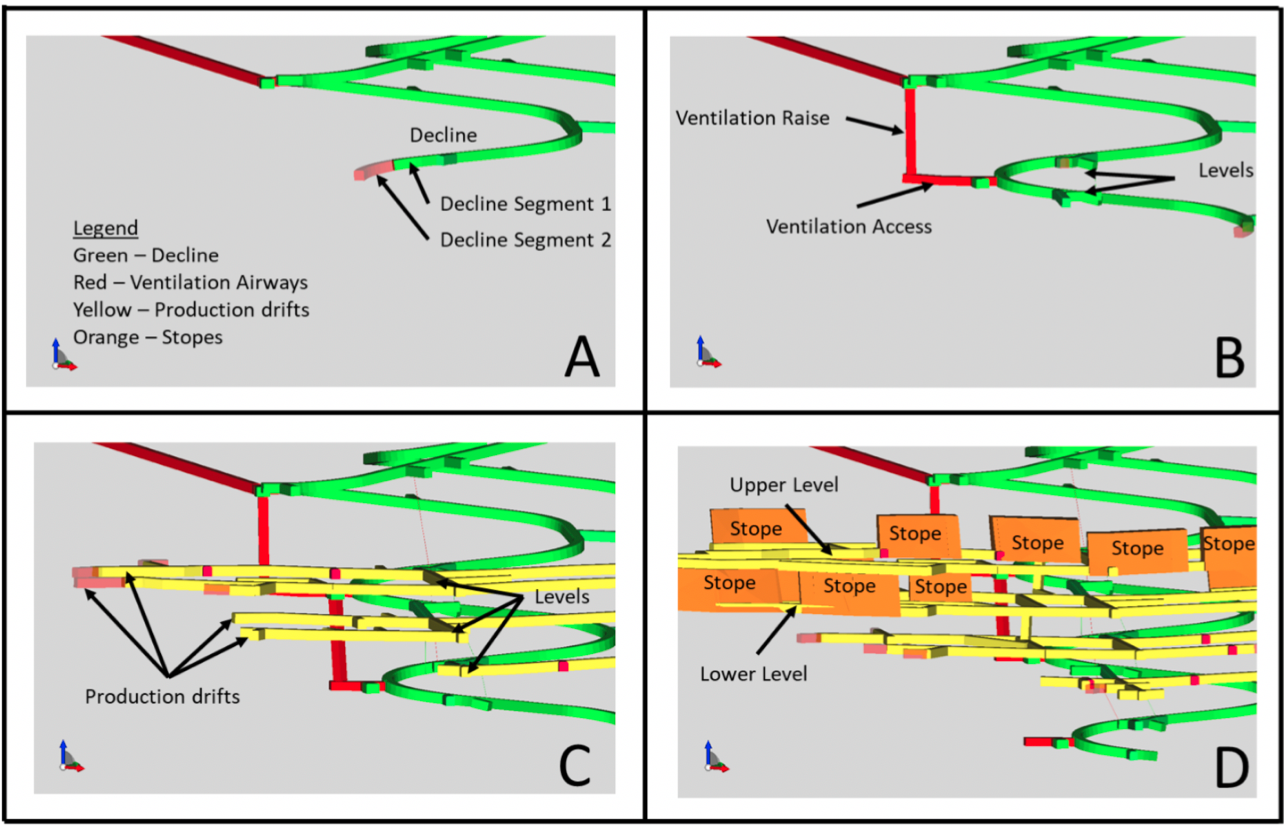}
    \caption{(A) Each design element is divided into activities (jobs), and each activity is associated with other activities through precedence. For example, in Quadrant A, declining segment two can only begin once segment one has been completed. (B) All working areas need to be ventilated in an underground mine to ensure a safe working environment. In Quadrant B, the decline is constrained from advancing until its previous section has been connected to the ventilation circuit through the ventilation access and raise. (C) The mine is divided into different levels used to access extraction areas, i.e., stopes, via production drifts excavated from the main decline. As depicted in Quadrant C, stope extraction at a given level can only begin once the production drifts have been fully completed. (D) In Quadrant D, the stopes on the lower level can only be extracted when stopes directly above them have been extracted to conform to precedences associated with a top-down mining method. Image credits: Akshay Chowdu and Andrea Brickey}
    \label{fig:example_underground}
\end{figure}

The second dataset comes from three real underground strategic mine planning problems. In these instances, the number of jobs is considerably larger than in the PSPLib instances, with thousands of jobs to schedule using a daily time resolution. Each job represents different activities required to extract the ore from the mine, including material extraction, drilling, vertical developments, and backfilling, among others. See Figure~\ref{fig:example_underground} for an example of how these jobs and their precedences are constructed. These jobs have positive profits when representing the extraction of ore, with a value depending on the ore grade of the material. Negative profits, in other cases, reflect the cost of these activities. 

To provide a realistic framework, we search for a daily schedule for the first five years of the mine. To eliminate unnecessary activities, we apply two preprocessing procedures to each instance. First, we compute the transitive closure $\hat{\mathcal{P}}$, removing all jobs requiring more than 1800 time periods to finish. Then, we apply the classical \emph{nested pit} heuristic based on the maximum closure problem: given a weight for each node (job) in $G$, we compute the closure of $G$ with maximum weight. For the weight $w$, we use a scaling parameter $\alpha\in [0,1]$ and we compute its value as $w_i = p_i$ if $p_i\leq 0$ or $w_i=\alpha\cdot p_i$ if $p_i > 0$. In this way, we obtain a smaller problem by diminishing the value of $\alpha$. We apply this heuristic to obtain a problem such that the total resource requirement is less than the availability after 1800 time periods. The discount factor for this experimental study is $0.02\%$, equivalent to $\sim 7.5\%$ annual discount rate.



\begin{table}[tbp]
    \centering
    \footnotesize
    \setlength{\tabcolsep}{4pt}
    \begin{tabular}{c|c|c|c||c|c||c|c}
   \multicolumn{4}{c||}{} & \multicolumn{2}{c||}{Cumulative resources} & \multicolumn{2}{c}{Non-cumulative resources}\\ \hline

    & & &  & Avg. Time & Avg. Gap  & Avg. Time & Avg. Gap \\
   Instance &  $|\mathcal{J}|$ & $|\mathcal{K}|$ & $\epsilon$  & (sec) & (\%)  & (sec) & (\%) \\ \hline 
\multirow{4}{*}{Agricola} & \multirow{4}{*}{2423} & \multirow{4}{*}{4}
& 1.0 & 0.6     & 6.16  & 0.5    & 7.15 \\
&&& 0.5 & 2.7   & 4.13  & 2.2    & 5.18 \\
&&& 0.2 & 21.4  & 2.00  & 16.8   & 3.19 \\
&&& 0.1 & 127.7 & 1.08  & 137.5  & 2.81 \\ \hline
\multirow{4}{*}{Catan}& \multirow{4}{*}{4383} & \multirow{4}{*}{8}
&1.0 & 2.3      & 7.68  & 3.8    & 10.61 \\
&&&0.5 & 16.9   & 6.72  & 15.6   & 9.12 \\
&&&0.2 & 1017.2 & 3.72  & 84.8   & 6.44  \\
&&&0.1 &13251.2 & 2.22  & 682.5  & 6.09  \\ \hline
\multirow{4}{*}{Dominion}& \multirow{4}{*}{8391} & \multirow{4}{*}{7}
& 1.0 & 9.9     &11.08  & 6.3    & 15.81 \\
&&& 0.5 & 106.7 & 8.23  & 22.6   & 13.52 \\
&&& 0.2 & 2349.0& 4.16  & 271.2  & 11.05 \\
&&& 0.1 &63679.3& 2.39  &4738.8  & 9.51 \\ \hline
    \end{tabular}
    \caption{Results for underground mine instances}
    \label{tab:res_Omp}
\end{table}

Table \ref{tab:res_Omp} shows the algorithm's results for the mine instances under different aggregations. First, we observe that when using $\epsilon=1$, it is possible to obtain a solution very quickly, with optimality gaps $\sim 10\%$ for the three instances. By reducing the value of $\epsilon$, we can get solutions closer to the optimum (optimality gap $< 3\%$) but with a computation time that grows exponentially. Note that the resulting gaps in practice are considerably smaller than the bounds provided by the theoretical approximation factor of the algorithm.
  
\begin{figure}[htbp]
    \centering
    \includegraphics[height=4.7cm]{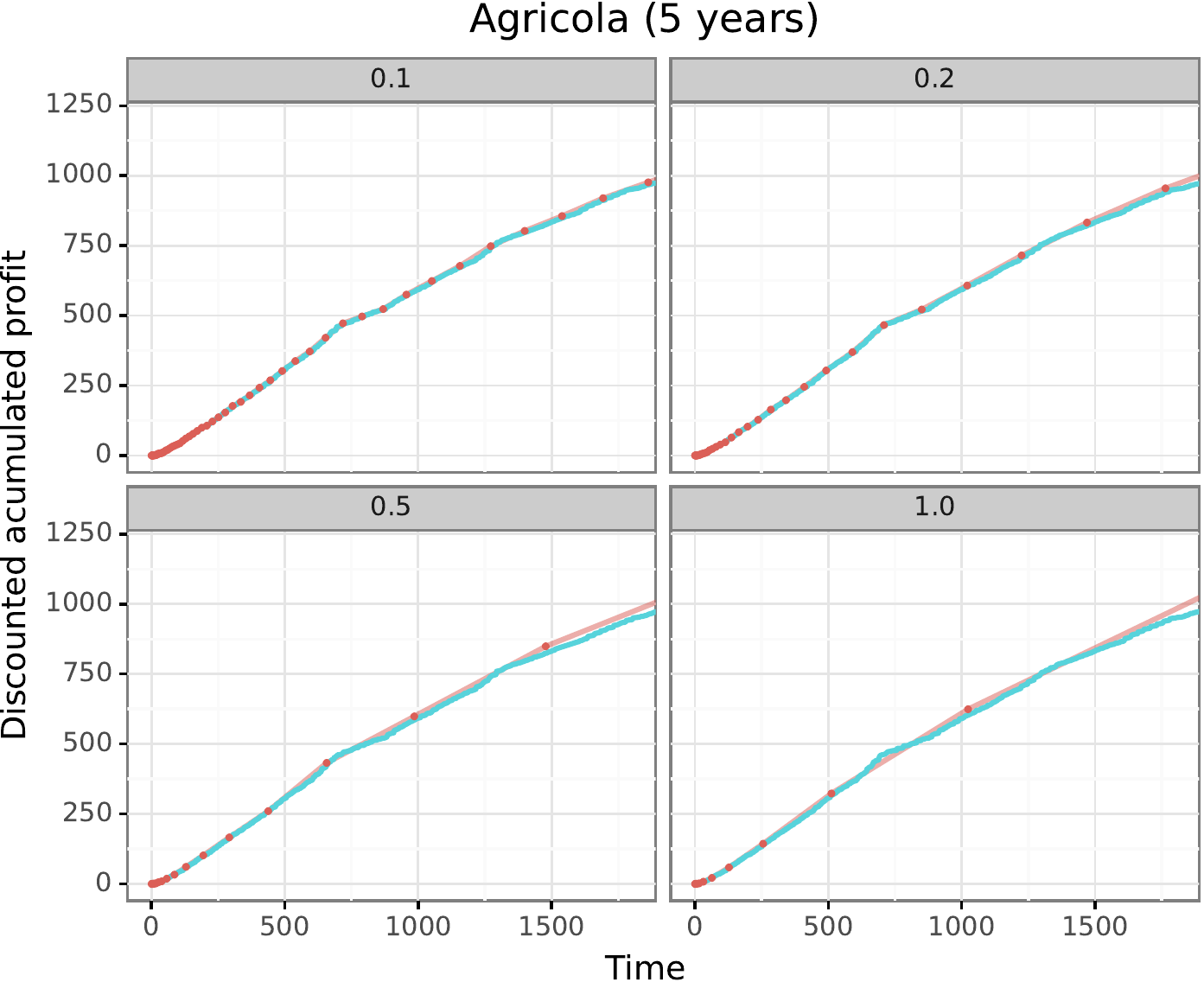}\hfill
    \includegraphics[height=4.7cm]{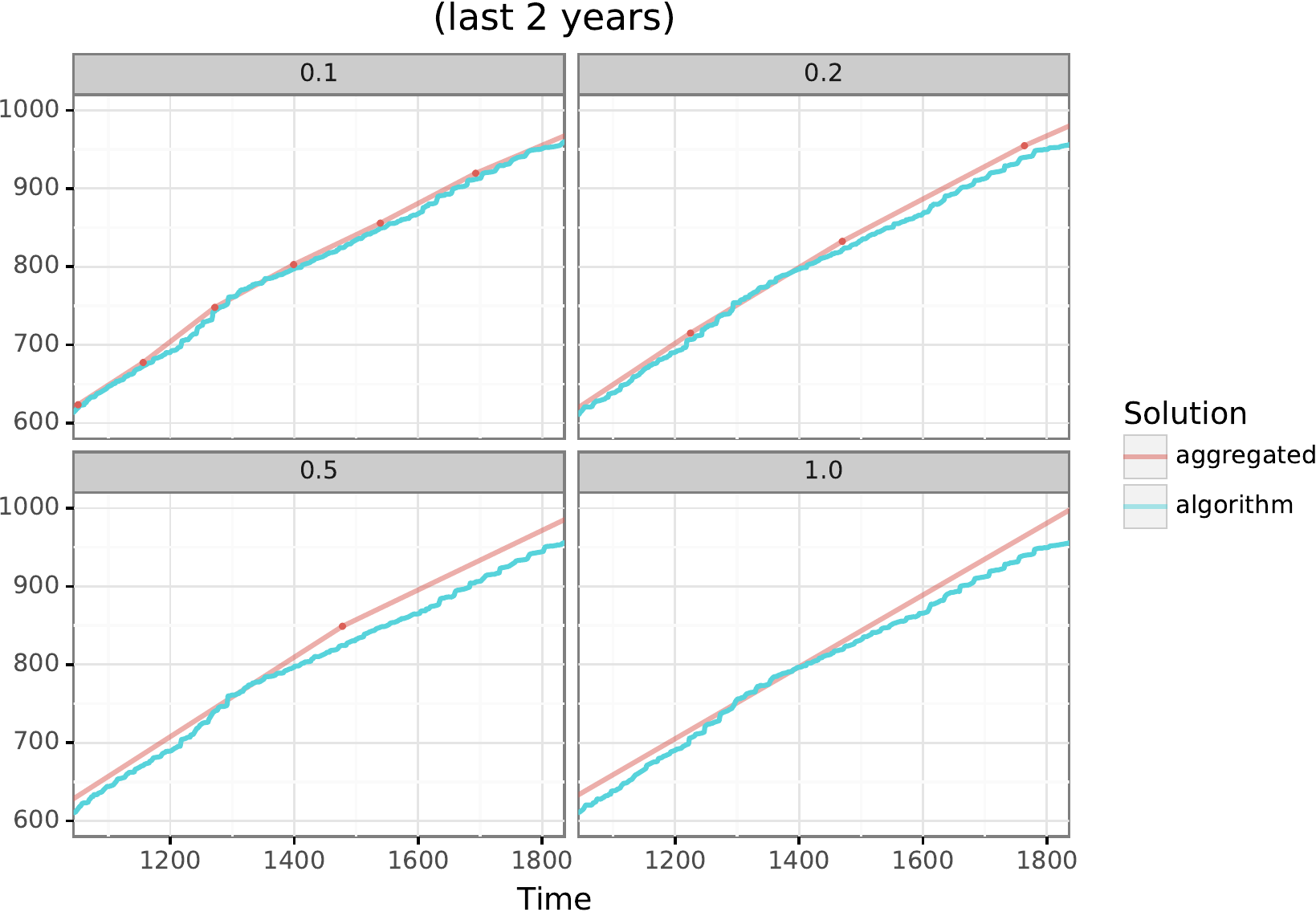}
    \includegraphics[height=4.7cm]{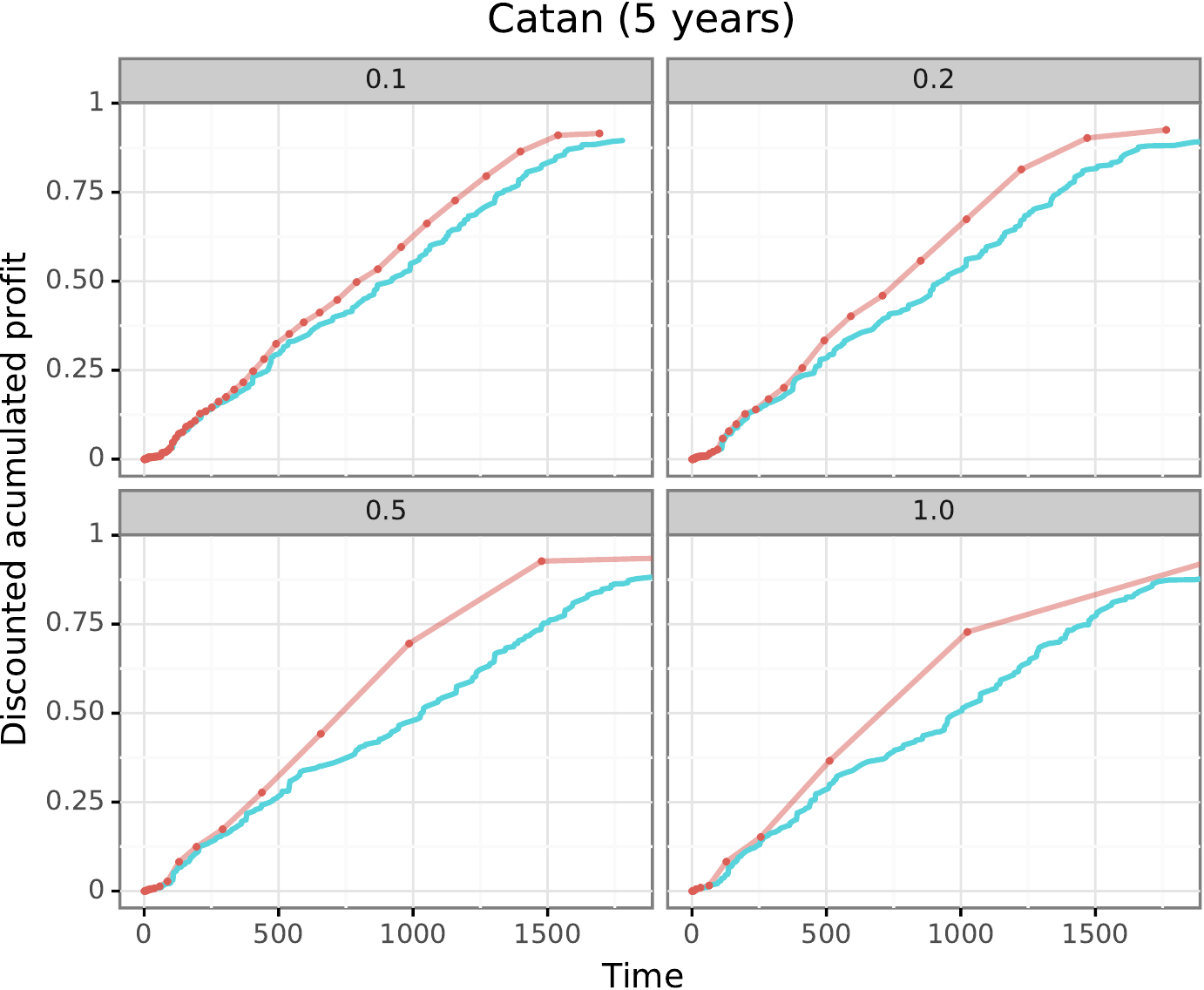}\hfill
    \includegraphics[height=4.7cm]{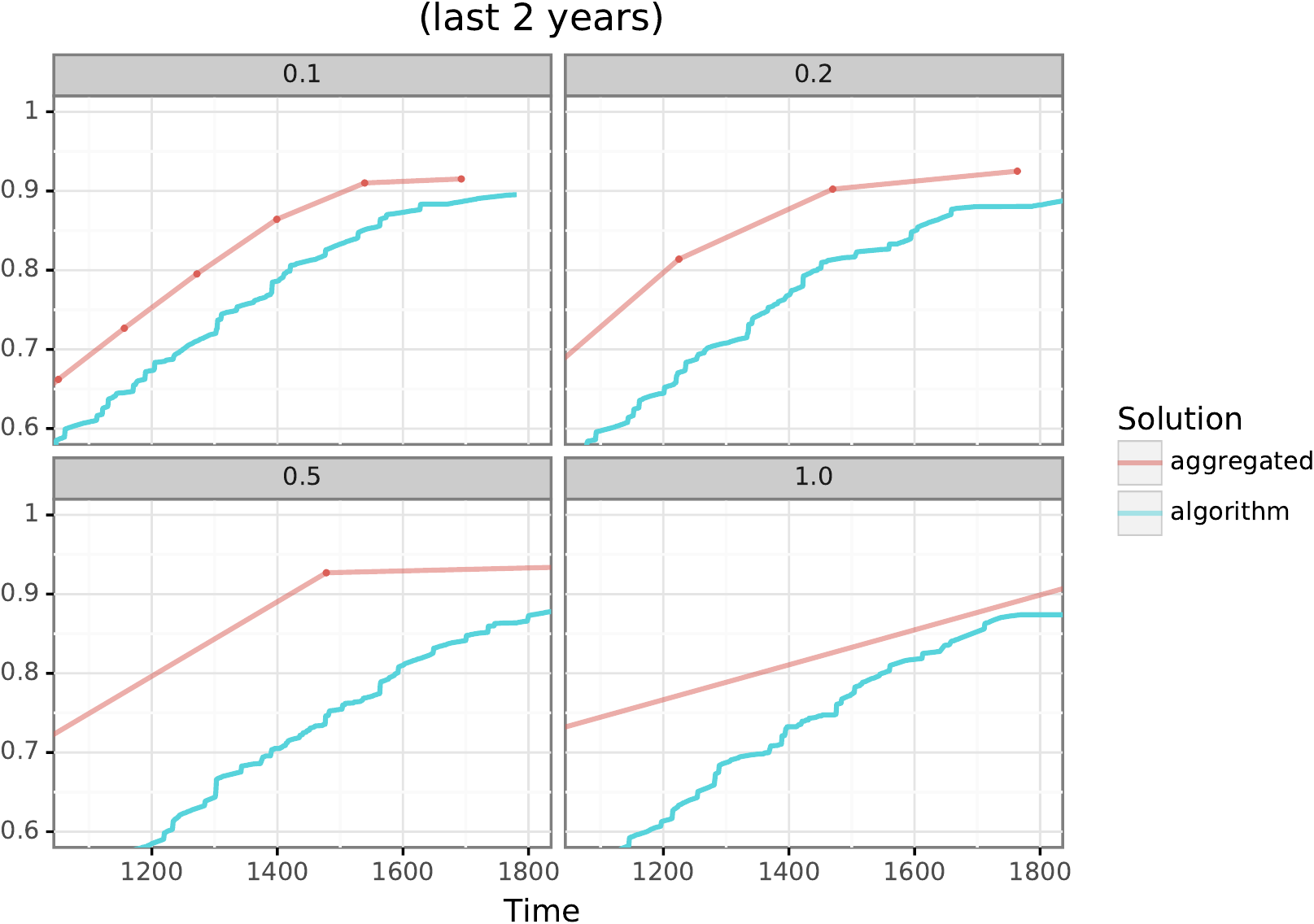}
    \includegraphics[height=4.7cm]{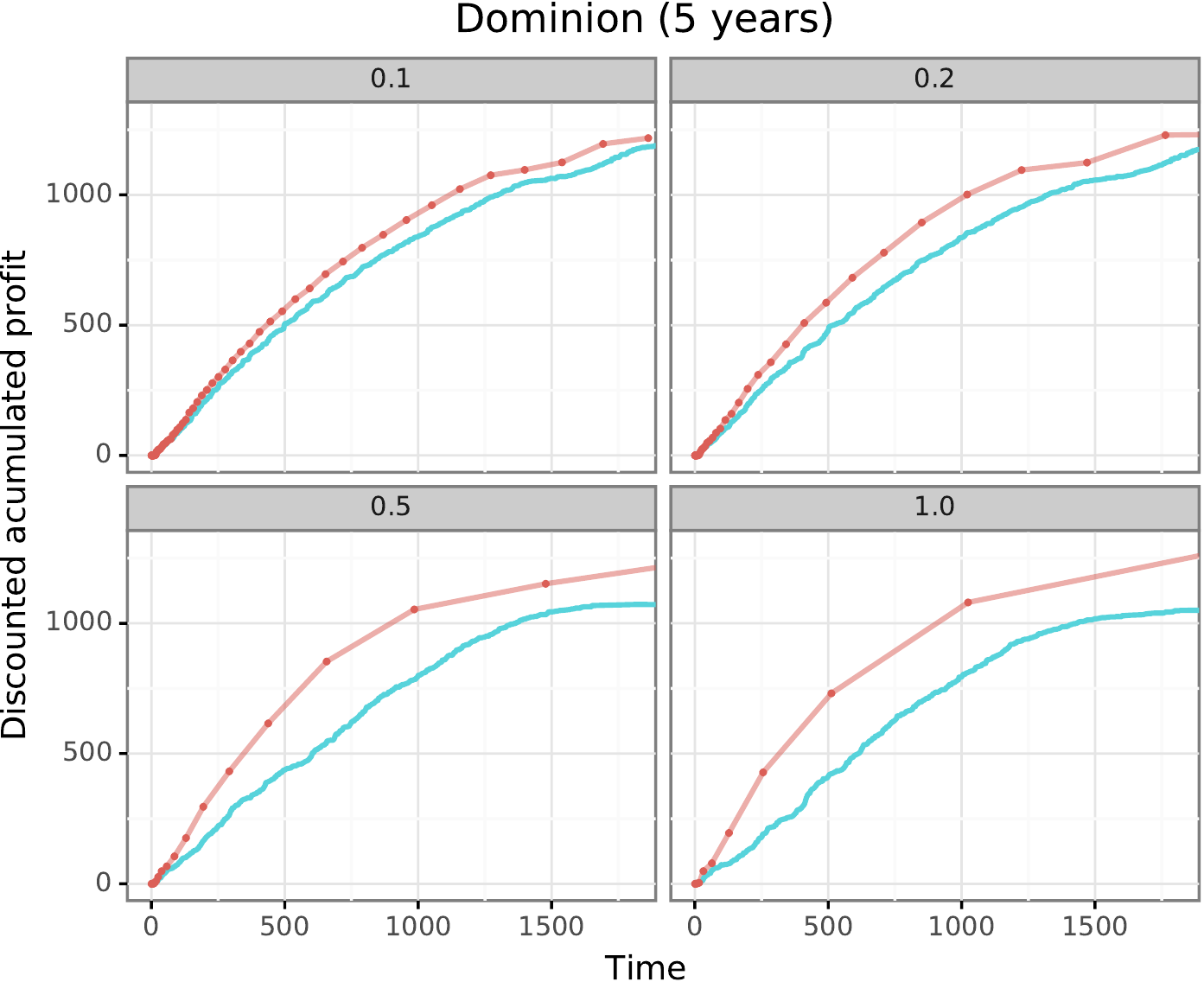}\hfill
    \includegraphics[height=4.7cm]{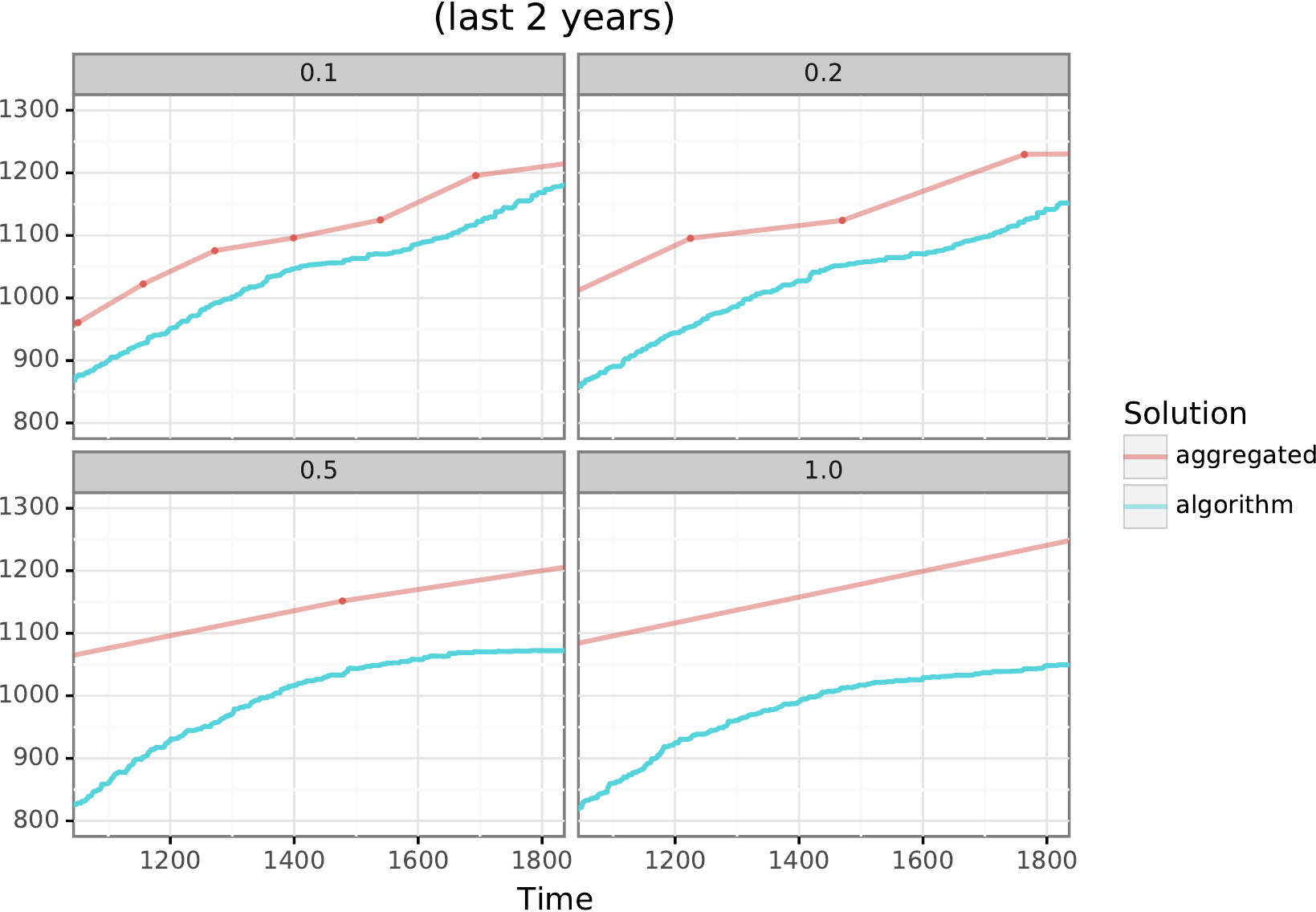}
    \caption{Discounted accumulated profit for different aggregation levels (left: five years, right: last two years)}
    \label{fig:res_Agricola}
\end{figure}

To study the behavior of the algorithm, in Figure~\ref{fig:res_Agricola}  we show the cumulative discounted profit at each time period for the solution obtained ($NPV(x^*)$) and for the aggregated MIP problem ($\widehat{NPV}(X^*)$) for each instance. For the latter case, the line interpolates the objective value at the end of each interval $\tau_s=(1+\epsilon)^s$, shown with small red dots.

It can be seen that the difference between the $NPV(x^*)$ and $\widehat{NPV}(X^*)$ is not produced in the initial periods but at the end of the time horizon. This difference is a desired property because the initial time periods have a more significant impact on the objective function due to the discount factors. Furthermore, this behavior is explained because the length of the time intervals in the initial periods is similar to the original one. For example, for $\epsilon=0.1$, the first 50 periods are divided into 41-time intervals, requiring a total of 79-time intervals to cover the 1800 periods, showing the exponential growth in the size of the intervals. The right figure shows the difference between  $\widehat{NPV}(X^*)$ and $NPV(x^*)$ for the last periods. It can be seen that a smaller aggregation level $\epsilon$ reduces this difference considerably, explaining the better gaps obtained. Interestingly, this reduction of optimality gaps is obtained mainly by a reduction on the upper bound $\widehat{NPV}(X^*)$, but also because this better estimation produces better feasible solutions $x^*$. For example, in Catan, the upper bound diminish a $4.7\%$, and the lower bound is increased by $1.1\%$ comparing $\epsilon=1$ versus $\epsilon=0.1$.

When the algorithm is applied considering the non-cumulative resource constraints, a similar behavior to the PSPlib dataset occurs. The resulting MIP problems are slightly faster to solve, but the obtained gaps are more significant than the cumulative case. This result is expected, as explained at the end of Section~\ref{sec:algo}. Nevertheless, the method allows us to obtain a solution with optimality gaps between $3\sim 9 \%$ in a few hours, which is adequately good in practice for problems of these sizes. Moreover, these more significant gaps are likely explained by a bad upper bound provided by $\widehat{NPV}(X^*)$. For example, using the BZ algorithm, we were able to solve the LP relaxation of Problem~\ref{eq:Orig} (replacing the resource constraint by~\eqref{eq:OrigResourcesRenew}) for the Dominion instance, which provide a different upper bound for the problem. This new bound required 19 hours of computation, and it shows that the solution obtained by our algorithm has an optimality gap of less than $5.6\%$ instead of the previously computed $9.5\%$ for this instance.

\section{Conclusions}
The RCPSP is a central problem in many industries and applications, many of which imply considerably large instances. This statement is particularly true in mining, where the problems have tens of thousands of tasks. Furthermore, sometimes the time resolution required is in the order of weeks, making problems prohibitively large for many current techniques.

Our work addresses this shortcoming by developing a new modeling procedure relevant to the RCPSP-DC case. The key insight is that scheduling errors in the first time periods have a much more significant impact on the final value of the solution than those made at the end of the program. Hence, if a decrease in time resolution is needed to handle larger instances, making it at the end of the program would affect the final solution less. Following this idea, we model the RCPSP-DC with geometrically increasing intervals, which gives us a high resolution at the beginning, but reduces it at the end to limit the size of the optimization problem, similarly to the cost reduction effect due to the net present value in the objective function. Furthermore, our approach allowed us to analyze the resulting approximation algorithm, proving that it has a bounded performance guarantee, which depends on the time horizon and the discount factor of the problem but not on its time granularity. The methodology also gives us a new tuning parameter for the approximation algorithm, allowing us to balance the optimality gap and the time required to compute a solution. To our knowledge, this is the first approximation algorithm developed for RCPSP-DC.

We also show that our approximation algorithm performs much better in practice than the computed theoretical bounds through experimental analysis on several different instances. Furthermore, we show that the proposed algorithm can handle real-life mining instances with tens of thousands of jobs and weekly resolutions. Finally, we remark that in our computational experiments, we are solving Problem~\ref{eq:Agg} using a standard MIP solver. However, for even larger problems or smaller values of $\epsilon$, we can utilize additional techniques to extend the usefulness of this methodology. For example, it is possible to solve the LP relaxation of Problem~\ref{eq:Agg} using decomposition algorithms and to apply Algorithm~\ref{alg:sort} considering the order for each job provided by the $\alpha$-intervals technique as explained in~\citep{Carrasco2018:rcasEjor}.

It is worth exploring how to use these modeling and analysis techniques for other RCPSP settings as an open future problem. Although the key idea of our approach is not that relevant when makespan is used in the cost function, it is still true that not the exact resolution is required for the whole schedule. Hence, using similar ideas of time-aggregation could result in new schemes to solve large RCPSP instances with bounded approximation ratios.

\appendix
\section{``By'' formulation for RCPSP-DC}\label{sec:by}
 
The ``By''-formulation of Problem~\ref{eq:Agg} can be obtained by replacing variables $X$ by a new variable $Y$ such that $Y_{jt}=\sum_{s\leq t} Y_{jt}$, or equivalently $X_{jt}=Y_{jt}-Y_{j,t-1}$ assuming $Y_{j0}=0$. With this change of variables, constraints~\eqref{eq:Aggelim}-\eqref{eq:Aggprec_at} can be rewritten as
\begin{align*}
Y_{jt} &= 0 &  t < \mathcal{I}(p_j),  j\in\mathcal{J} \\
    Y_{j,t-1} &\leq Y_{jt}  &  j\in\mathcal{J}, t\in\mathcal{T}_I \\
   Y_{jt} &\leq Y_{j,\mathcal{I}(\tau_{t}-\Delta_{jk})} &  k\in\hat{\mathcal{P}}_j, j\in\mathcal{J}, t\in\mathcal{T}_I \\
    Y_{jt} &\in \{0,1\} &  j\in\mathcal{J},  t\in\mathcal{T}_I 
\end{align*}
and the resource constraints as
\[     \sum_{j=1}^{N} \sum_{u=1}^{\mathcal{I}(\tau_t+p_j)-1} q_{jk}  (p_j - (\tau_u-\tau_t)^+)  \left(Y_{ju}-Y_{j,u-1}\right) \leq  \sum_{s=1}^{\lceil\tau_t\rceil}R_{ks} \quad t\in\mathcal{T}_I , k\in\mathcal{K} \]
and the objective function
\[\max \sum_{\substack{j=1\\f_j>0}}^{N}\sum_{s=1}^{T_I}  \frac{f_{j}}{(1+r)^{\tau_{s-1}}}  \left(Y_{js}-Y_{j,s-1}\right) 
+ \sum_{\substack{j=1\\f_j<0}}^{N}\sum_{s=1}^{T_I}  \frac{f_{j}}{(1+r)^{\tau_{s}}} \left(Y_{js}-Y_{j,s-1}\right).\]

As result, the model has $\mathcal{O}(|\mathcal{T}_I|\cdot |\mathcal{J}|\cdot \max_{j\in J} |\hat{\mathcal{P}_j}|)$ precedence constraints of the form $Y_{jt}\leq Y_{ks}$, but only $\mathcal{O}(|\mathcal{T}_I|\cdot |\mathcal{K}|)$ additional resource constraints. This structure allows us to apply a decomposition method, like Lagrangian relaxation or the BZ algorithm. The latter technique penalizes the resource constraints on the objective function (similarly to Lagrangian relaxation) for obtaining a maximum closure problem that can be solved efficiently using max-flow algorithms, allowing it to solve its LP relaxation efficiently.

\bibliographystyle{abbrvnat}
\bibliography{rcpsp}

\end{document}